\documentclass[a4paper,12pt,leqno]{amsart}
\usepackage{color}
\usepackage{url}
\usepackage{amsmath}
\usepackage{amsthm}
\usepackage{amssymb}
\usepackage{mathrsfs}
\usepackage{here}
\usepackage{orcidlink}

\usepackage{subcaption}
\usepackage{tikz}
\usetikzlibrary{intersections, calc, arrows.meta, patterns}

\allowdisplaybreaks
\usepackage{mathtools}

\numberwithin{equation}{section}

\theoremstyle{definition}
\newtheorem{Definition}{Definition}[section]
\newtheorem{Theorem}[Definition]{Theorem}
\newtheorem{Proposition}[Definition]{Proposition}
\newtheorem{Lemma}[Definition]{Lemma}
\newtheorem{Corollary}[Definition]{Corollary}
\newtheorem{Remark}[Definition]{Remark}

\title{Stability of the critical constant steady state of a Keller--Segel model}
\date{}
\author[N. Miyake]{Nobuhito Miyake}
\address[N. Miyake]{
Faculty of Mathematics, Kyushu University,
Fukuoka-shi, Fukuoka \hbox{819-0395}, Japan
}
\email{miyake@math.kyushu-u.ac.jp}
\author[H. Wakui]{Hiroshi Wakui}
\address[H. Wakui]{
Faculty of Engineering, 
University of Fukui, 
Fukui-shi,
Fukui \hbox{910-8507}, Japan}
\email{hwakui@u-fukui.ac.jp}
\author[T. Yamada]{Tetsuya Yamada}
\address[T. Yamada]{National Institute of Technology(KOSEN), Fukui College, Sabae-shi, Fukui \hbox{916-8507}, Japan}
\email{yamada@fukui-nct.ac.jp}

\begin{document}

\newcommand{\norm}[1]{\left|\!\middle|\!\middle|#1\middle|\!\middle|\!\right|}
\newcommand{\Short}[1]{\left\langle#1\right\rangle_{S}}
\newcommand{\Long}[1]{\left\langle#1\right\rangle_{L}}

\maketitle

\begin{abstract}
In this paper, we prove the asymptotic stability of the critical constant steady state for a simplified parabolic--elliptic Keller--Segel system in $\mathbb{R}^N$ ($N \ge 3$), 
which admits a one-parameter family of constant steady states. 
Although the stability threshold for constant steady states is known, the critical case has remained open. 
We also show that the convergence rate in the critical case differs from the rates obtained for previously studied subcritical constant steady states.
\end{abstract}

%============================================================================%
%============================================================================%
\section{Introduction}
\label{sec:intro}
%============================================================================%
%============================================================================%

In this paper, we consider the stability of the critical constant stationary solution to the following Cauchy problem of a simplified parabolic--elliptic Keller--Segel system:
\begin{equation}
\label{eq:KS}
\tag{KS}
	\left\{
	\begin{alignedat}{2}
		\partial_t u 
		- 
		\Delta u 
		+ 
		\nabla \cdot (u \nabla \psi)
		&
		=0,
		&&\quad\text{in}\quad 
		\mathbb{R}^N\times(0, T),\\
		-\Delta \psi+\psi
		&=u
		&&\quad\text{in}\quad 
		\mathbb{R}^N \times (0, T),\\
		u(\cdot, 0)
		&=u_0(\cdot)
		&&\quad\text{in}\quad 
		\mathbb{R}^N,
	\end{alignedat}
	\right.
\end{equation}
where $N \ge 1$ and $T>0$.
When 
$
	u_0 \in L^p(\mathbb{R}^N)
$ 
for some 
$ 
	p \in [1,\infty)
$, 
Problem~\eqref{eq:KS} has been extensively studied from the viewpoints of existence and uniqueness,
blow-up criteria, and the asymptotic behavior of global-in-time solutions; see, e.g., \cite{BCK2017, CPZ2004, F2022, KS2010, KS2011, KS2008, KSY2012, R2011}.
Recently, 
in order to consider a class of solutions that includes constant stationary solutions
$
	(u,\psi)
	=
	(A,A)
$ 
$(A\in\mathbb{R})$, 
several works have employed uniformly local Lebesgue spaces (\cite{CKKW2021, S2021}).
In particular, Cygan, Karch, Krawczyk, and the second author studied the stability of constant stationary solutions 
$
	(u, \psi)
	=
	(A, A)
$ 
for $A\in\mathbb{R}$ and showed the following results (\cite[Theorems~2.3 and 2.4, Remark~4.8]{CKKW2021}):
\begin{enumerate}
	\item
		Let $A \in (-\infty, 1)$ and assume that
		\[
			\left\{
			\begin{split}
				 p = \, & 1  &&\text{if} \qquad N=1,\\
				 \frac{N}2 < p \, & < N &&\text{if} \qquad N \ge 2.
			\end{split}
			\right.
		\]
		Assume that $u_0-A$ is sufficiently small in $L^p(\mathbb{R}^N)$.
		Then problem~\eqref{eq:KS} admits a global-in-time solution such that $u-A$ converges to $0$ in some suitable sense.
	\item
		Let $A\in(1, \infty)$ and
		\[
			\left\{
			\begin{split}
				1 <  p < \, & \infty  &&\text{if} \qquad N=1,\\
				\frac32 \le  p < \, & \infty && \text{if} \qquad N=2,\\
				\frac{N}2 < p < \, & \infty &&\text{if} \qquad N \ge 3.
			\end{split}
			\right.
		\]
		Then the stationary solution $(u,\psi)=(A, A)$ of \eqref{eq:KS} is not stable in the Lyapunov sense under small perturbations from $L^p(\mathbb{R}^N)$.
\end{enumerate}
These results imply that $(u,\psi)=(A,A)$ is asymptotically stable if $A<1$ and is unstable if $A>1$.
In other words, 
$A=1$ is the threshold that determines whether the stationary solution 
$
	(u, \psi)
	=
	(A,A)
$ 
is stable.
See also \cite{WYarXiv} for a related study on the stability of constant stationary solutions to 
an attraction--repulsion chemotaxis model.

\subsection{Main result}

The main objective of this paper is to establish that the ``critical'' stationary solution 
$
	(u, \psi)
	=
	(1, 1)
$ 
of problem~\eqref{eq:KS} with $N\ge3$ is asymptotically stable.
In order to state our main theorem, we introduce some notation.
Set
\begin{alignat}{2}
	\label{eq:kernel}
		G(x, t)
		\coloneqq \, & 
		(4\pi t)^{-\frac{N}{2}}
			\mathrm{e}^{-\frac{|x|^2}{4t}}
		&&\quad \text{for} \quad 
		(x, t) 
		\in 
		\mathbb{R}^N \times (0, \infty),\\
	\label{eq:Bessel}
		K(x)
		\coloneqq \, &  
		\int_0^\infty 
			\mathrm{e}^{-s}G(x,s)
		\, \mathrm{d}s
		&&\quad \text{for} \quad 
		x
		\in
		\mathbb{R}^N \setminus \{ 0 \}.
\end{alignat}
For $p\in[1,\infty]$ we define the uniformly local Lebesgue space $L^{p}_{\rm ul}(\mathbb{R}^N)$ as
\[
	L^{p}_{\rm ul}(\mathbb{R}^N)
	\coloneqq 
	\left\{
		f \in L^1_{\rm loc}(\mathbb{R}^N)
		\, \middle| \, 
		\|
			f
		\|_{L^{p}_{\rm ul}(\mathbb{R}^N)}
		\coloneqq 
		\sup_{z \in \mathbb{R}^N}
		\|
			f
		\|_{L^p(B_1(z))}
		<
		\infty
	\right\},
\]
where $B_r(z)$ denotes the ball of radius $r>0$ centered at $z\in\mathbb{R}^N$.
Moreover, 
we also define the closed subspace 
$
	\mathcal{L}^{p}_{\rm ul}(\mathbb{R}^N)
$ 
of 
$
	L^{p}_{\rm ul}(\mathbb{R}^N)
$ 
as
\[
	\mathcal{L}^p_{\rm ul}(\mathbb{R}^N)
	\coloneqq 
	\overline{BUC(\mathbb{R}^N)}^{\| \cdot \|_{L^{p}_{\rm ul}(\mathbb{R}^N)}},
\]
where $BUC(\mathbb{R}^N)$ is the set of bounded and uniformly continuous functions on $\mathbb{R}^N$.
In this paper, 
we consider solutions to problem~\eqref{eq:KS} in the following sense (see also the relation \eqref{eq:poisson} below):

\begin{Definition}
\label{def:KSsol}
	Let $N\ge3$, 
	$T\in(0, \infty]$, 
	and 
	$
		u_0
		\in 
		L^\frac{N}{2}_{\rm ul}(\mathbb{R}^N)
	$.
	We say that $u$ is a solution to problem~\eqref{eq:KS} in $\mathbb{R}^N\times(0, T)$ if 
	$
		u 
		\in 
		C([0, T); L^\frac{N}{2}_{\rm ul}(\mathbb{R}^N))
	$ 
	with $u(0)=u_0$ and $u$ satisfies
	\[
		u(t)
		=
		G(t) \ast u_0
		-
		\int^t_0
			\nabla 
			\cdot 
			\left[
				G(t-s) \ast (u(s)\nabla K \ast u(s))
			\right]
		\, \mathrm{d}s
	\]
	in $L^\frac{N}{2}_{\rm ul}(\mathbb{R}^N)$ for $t\in(0, T)$.
	If $T=\infty$, 
	$u$ is called a global-in-time solution to problem~\eqref{eq:KS}.
\end{Definition}
To describe the long-time and short-time behavior in a unified manner, we define
$\Long{t}$ and $\Short{t}$ for $t>0$ by
\[
	\Long{t}
	\coloneqq
	\sqrt{1+t^2},
	\qquad 
	\Short{t}
	\coloneqq 
	\dfrac{t}{\Long{t}}=\dfrac{t}{\sqrt{1+t^2}}.
\]
These quantities interpolate smoothly between the short-time regime ($0 < t \ll 1$) and the long-time regime ($t \gg 1$). 
For simplicity, 
we set
$
	\|
		f
	\|_{p}
	\coloneqq 
	\|
		f
	\|_{L^p(\mathbb{R}^N)}
$
for $p\in[1,\infty]$.
The main result in this paper is as follows:
\begin{Theorem}
\label{thm:main_KS}
	Let $N\ge3$.
	Then there exist constants $\varepsilon_*>0$ and $C>0$ with the following property:
	if 
	$
		u_0 
		\in 
		\mathcal{L}^\frac{N}{2}_{\rm ul}(\mathbb{R}^N) 
	$ 
	satisfies
	\begin{equation}
	\label{eq:main_KS_asm}
		\|
			u_0
			-
			1
		\|_{\frac{N}{2}}
		\le 
		\varepsilon_*,
	\end{equation}
	then there exists a global-in-time solution $u$ to problem~\eqref{eq:KS} such that 
	\[
		u, \nabla u
		\in 
		C((0, \infty); \mathcal{L}^{r}_{\rm ul}(\mathbb{R}^N))
	\] 
	and
	\[
		\Short{t}^{\frac{N}{2}(\frac{2}{N}-\frac{1}{r})}
		\Long{t}^{\frac{N}{4}(\frac{2}{N}-\frac{1}{r})}
		\|
			u(t)
			-
			1
		\|_{r}
		+
		\Short{t}^{\frac{N}{2}(\frac{2}{N}-\frac{1}{r})+\frac{1}{2}}
		\Long{t}^{\frac{N}{4}(\frac{2}{N}-\frac{1}{r})+\frac{1}{4}}
		\|
			\nabla u(t)
		\|_{r}
		\le
		C
		\|
			u_0
			-
			1
		\|_{\frac{N}{2}}
	\]
	for $r \in [N/2, \infty]$ and $t>0$.
	In particular, 
	$u(t)-1$ converges to $0$ in $W^{1,r}(\mathbb{R}^N)$ as $t \to \infty$
	for $r \in (N/2, \infty]$.
\end{Theorem}

\begin{Remark}
	In Appendix~\ref{sec:unique}, 
	we prove conditional uniqueness of the solution to problem~\eqref{eq:KS} in the sense of Definition~\ref{def:KSsol} (see Proposition~\ref{prop:unique}).
	By Proposition~\ref{prop:unique}, 
	the solution constructed in Theorem~\ref{thm:main_KS} is the unique solution in some suitable function space.
\end{Remark}

To prove Theorem~\ref{thm:main_KS}, 
we follow the argument in \cite[Theorem~2.3]{CKKW2021} and consider the perturbation problem associated with \eqref{eq:KS} around the constant stationary solution 
$
	(u, \psi)
	=
	(1, 1)
$.
More precisely, 
by setting 
$
	(v, \varphi)
	\coloneqq 
	(u-1, \psi-1)
$, 
a pair of functions
$(v, \varphi)$ satisfies the following Cauchy problem:
\begin{equation}
\label{eq:KS_A}
\tag{P}
	\left\{
	\begin{alignedat}{2}
		\partial_t v
		-
		\Delta v
		+
		\Delta \varphi
		+
		\nabla \cdot (v \nabla \varphi)
		&=0,
		&&\quad \text{in} \quad 
		\mathbb{R}^N \times (0, T),\\
		-
		\Delta \varphi
		+
		\varphi
		&=v
		&&\quad \text{in} \quad 
		\mathbb{R}^N \times (0, T),\\
		v(\cdot, 0)
		&=v_0(\cdot)
		&&\quad\text{in}\quad \mathbb{R}^N.
	\end{alignedat}
	\right.
\end{equation}
The main difference between Theorem~\ref{thm:main_KS} and \cite[Theorem~2.3]{CKKW2021} comes from the behavior of the Fourier symbol for the linearized equation of \eqref{eq:KS}.
Indeed, the linearized equation of \eqref{eq:KS} around 
$
	(u, \psi)
	=
	(A, A)
$ 
is 
$
	\partial_t w
	-
	\Delta w
	+
	A\Delta K \ast w
	=
	0
$ 
and the Fourier symbol for $-\Delta +A\Delta K\ast $ is given by 
\[
	h_A(\xi)
	\coloneqq 
	|\xi|^2
	-
	\dfrac{
			A|\xi|^2
		}{
			1+|\xi|^2
		}
	=
	\dfrac{
			(1-A)|\xi|^2+|\xi|^4
		}{
			1+|\xi|^2
		}.
\]
Since $h_A(\xi) \sim |\xi|^2$ if $A<1$, 
it is possible to treat the linearized equation like the heat equation.
On the other hand, 
if $A=1$, 
it holds that
\[
	h_{A}(\xi)
	\sim 
	\left\{
	\begin{alignedat}{2}
			&|\xi|^4
			&&\quad \text{for} \quad 
			|\xi|\ll 1,\\
			&|\xi|^2
			&&\quad \text{for} \quad 
			|\xi|\gg 1.
	\end{alignedat}
	\right.
\]
Thus, 
the result of \cite{CKKW2021} cannot be directly extended to the critical case $A = 1$.
Moreover, 
since the low-frequency part of $h_A(\xi)$ dominates the asymptotic behavior of solutions at $t\to\infty$, 
the asymptotic behavior of solutions to problem~\eqref{eq:KS_A} is different from the case 
$
	A
	\in
	[0, 1)
$.
The critical case $A = 1$ exhibits a fundamentally different structure, 
as the low-frequency behavior of the Fourier symbol is governed by $|\xi|^4$ rather than $|\xi|^2$. 
This invalidates existing approaches based on heat-like dissipation. 
We overcome this difficulty by developing a new framework combining frequency decomposition and refined time-weighted estimates.

To conclude this subsection, we briefly explain why the initial
perturbation is assumed to be in $L^{\frac{N}{2}}(\mathbb{R}^N)$ in
Theorem~\ref{thm:main_KS}. 
Since $h_1(\xi)=|\xi|^4/(1+|\xi|^2)$ is the linear symbol of \eqref{eq:KS_A}, and since the symbol of $\nabla K$ is
$i\xi/(1+|\xi|^2)$, the low-frequency part of \eqref{eq:KS_A} is formally governed
by
\begin{equation}\label{eq;LS}
  \partial_t v+(-\Delta)^2v+\nabla\cdot(v\nabla v)=0,
\end{equation}
whereas the high-frequency part is formally governed by
\begin{equation}\label{eq;GS}
  \partial_t v-\Delta v+\nabla\cdot
  \bigl(v\nabla(-\Delta)^{-1}v\bigr)=0.
\end{equation}
Since both model equations \eqref{eq;LS} and \eqref{eq;GS} leave 
$
	L^{\frac{N}2}(\mathbb{R}^N)
$ 
invariant under their natural scalings,
$
	L^{\frac{N}2}(\mathbb{R}^N)
$ 
is a natural space for the initial perturbation.

\subsection{Organization} 

This paper is organized as follows.
In Section~\ref{sec:linearized}, 
we derive estimates for the fundamental solution $F$  to $\partial_t w-\Delta w+\Delta K\ast w=0$, which is defined by \eqref{fn;F} below.
More precisely, 
we separate $F$ into two parts, the high- and low-frequency parts, and derive time-decay estimates for them.
Based on these estimates, 
we give bilinear estimates associated with problem~\eqref{eq:KS_A} and prove Theorem~\ref{thm:main} in Section~\ref{sec:perturbation}.
As we will explain in Section~\ref{sec:perturbation}, 
Theorem~\ref{thm:main_KS} is a direct consequence of Theorem~\ref{thm:main}.
In Appendix~\ref{sec:unique}, 
we prove a conditional uniqueness result for problem~\eqref{eq:KS}.

\subsection{Notation} 

Before closing this section, we 
introduce notation used throughout this paper. 
Let 
$
	\mathbb{Z}_{+}
$ 
denote
the set of nonnegative integers. 
For 
$ 
 x=(x_1,x_2,\ldots,x_N) \in \mathbb{R}^N
$ 
and 
$
 \alpha
 =
 (
  \alpha_1, 
  \alpha_2,
  \ldots,
  \alpha_N
 )
 \in 
 \mathbb{Z}_+^N
$, 
set 
$ 
  x^\alpha
  \coloneqq 
  x_1^{\alpha_1}
  x_{2}^{\alpha_2}
  \cdots 
  x_{N}^{\alpha_N}
$ 
and 
$
 |\alpha|
 \coloneqq
  \alpha_1
  +
  \alpha_2
  +
  \cdots
  +
  \alpha_N
$. 
We denote
$
  \partial_t
  \coloneqq
 \frac{\partial}{\partial t}
$ 
and 
$
 \partial_j
 \coloneqq
 \frac{\partial}{\partial x_j}
$,
and define
$
 \partial_x^\alpha
 \coloneqq
 \partial_1^{\alpha_1}
 \partial_2^{\alpha_2}
 \cdots
 \partial_N^{\alpha_N}
$ 
for 
$
 \alpha
 =
 (
  \alpha_1,
  \alpha_2,
  \ldots, 
  \alpha_N
 )
 \in 
 \mathbb{Z}_+^N
$. 
Let $C$ be a positive constant which may change from line to line. In particular, we write $C(\gamma_1,\ldots,\gamma_k)$ for a positive constant depending on $\gamma_1, \ldots, \gamma_k$.

%============================================================================%
%============================================================================%
\section{{Linearized equation}}
\label{sec:linearized}
%============================================================================%
%============================================================================%
In this section, 
we study properties of the fundamental solution to the following linear equation, 
which corresponds to problem~\eqref{eq:KS_A}:
\begin{equation}
\label{eq:L}
		\partial_t w
		-
		\Delta w
		+
		\Delta K \ast w
		=
		0
		\quad \text{in} \quad 
		\mathbb{R}^N \times (0, \infty).
\end{equation}
Let $F$ be the fundamental solution to \eqref{eq:L}, 
that is,
\begin{equation}\label{fn;F}
	F(x, t)
	\coloneqq 
	(2\pi)^{-N}
	\int_{\mathbb{R}^N}
		\mathrm{e}^{-th(\xi)+ix\cdot \xi}
	\, \mathrm{d}\xi,
	\quad \text{where} \quad 
	h(\xi)
	\coloneqq 
	\dfrac{|\xi|^4}{1+|\xi|^2}.
\end{equation}
Let $\Phi \in C^\infty_{\rm c}(\mathbb{R}^N)$ satisfy
\[
	0
	\le
	\Phi
	\le 
	1
	\quad \text{on} \quad 
	\mathbb{R}^N,
	\quad
	\Phi
	\equiv 
	1
	\quad \text{on} \quad 
	B_1(0),\quad 
	\Phi 
	\equiv 
	0
	\quad \text{on} \quad 
	\mathbb{R}^N \setminus B_2(0).
\]
Setting 
$
	\Phi_1(\xi)
	\coloneqq 
	1-\Phi(\xi)
$ 
and
$
	\Phi_2(\xi)
	\coloneqq 
	\Phi(\xi)
$, 
we have 
$
	F(x, t)
	=
	F_1(x, t)
	+
	F_2(x, t)
$, 
where
\begin{equation}\label{fn;Fj}
	F_j(x, t)
	\coloneqq 
	(2\pi)^{-N}
	\int_{\mathbb{R}^N}
		\Phi_{j}(\xi)
		\mathrm{e}^{-th(\xi)+ix\cdot\xi}
	\, \mathrm{d}\xi
\end{equation}
for $j \in \{ 1, 2 \}$. 
The aim of this section is to prove a variant of the $L^p$-$L^q$ estimates.
For simplicity, 
we set 
$\mathcal{L}^p\coloneqq \overline{\mathcal{S}}^{\|\cdot\|_{p}}$ for $p\in [1,\infty]$, that is, 
\[
	\mathcal{L}^p
	=
	L^p(\mathbb{R}^N) 
	\quad \text{if} \quad 1\le p<\infty, 
	\quad 
	\mathcal{L}^\infty
	=
	\{
		f \in C(\mathbb{R}^N) 
		\mid  
		|f(x)| \to 0 \, \text{as} \, |x| \to \infty
	\}.
\]
Here, $\mathcal{S}$ is the set of rapidly decreasing functions on $\mathbb{R}^N$.
\begin{Proposition}
\label{prop:semigroup}
	Let 
	$N \ge 1$, 
	$\alpha\in\mathbb{Z}_+^N$, 
	$1\le p\le q\le \infty$, 
	and $f\in \mathcal{L}^p$.
	\begin{enumerate}
		\item
			It holds that $\partial_x^\alpha [F(t)\ast f]\in \mathcal{L}^{q}$ for $t>0$.
			Moreover, there exists $C=C(N, p, q, \alpha)>0$ such that
			\begin{equation}\label{lqlp_est;F}
				\|
					\partial_{x}^\alpha F(t) \ast f
				\|_{q}
				\le 
				C
				\Short{t}^{-\frac{N}{2}(\frac{1}{p}-\frac{1}{q})-\frac{|\alpha|}{2}}
				\Long{t}^{-\frac{N}{4}(\frac{1}{p}-\frac{1}{q})-\frac{|\alpha|}{4}}
				\|
					f
				\|_{p}
			\end{equation}
			for $t>0$.
		\item
			Assume that $N\ge2$ and $|\alpha|\ge1$.
			Then $\partial_x^\alpha [F_j(t)\ast f]\in \mathcal{L}^{q}$ for $j\in\{1,2\}$ and $t>0$.
			Moreover, there exists $C=C(N, p, q, \alpha)>0$ such that
			\begin{equation}\label{lqlp_est;Fj}
				\|
					\partial_{x}^\alpha F_j(t) \ast f
				\|_{q}
				\le 
				C
				t^{-\frac{N}{2j}(\frac{1}{p}-\frac{1}{q})-\frac{|\alpha|}{2j}}
				\|
					f
				\|_{p}
			\end{equation}
			for 
			$j \in\{1,2\}$ 
			and $t>0$.
	\end{enumerate}
\end{Proposition}
The estimates \eqref{lqlp_est;F} and \eqref{lqlp_est;Fj} in Proposition~\ref{prop:semigroup} immediately follow from the H\"{o}lder inequality, the Young inequality, the density of $\mathcal{S}$, and the following two lemmas:
\begin{Lemma}
\label{lem:L^infty}
	Let $N\ge1$ and $\alpha \in \mathbb{Z}^N_{+}$.
	Then there exists $C=C(N, \alpha)>0$ such that
	\[
		\|
		\partial_x^\alpha 
			F(t)
		\|_{\infty}
		\le 
		C
		\Short{t}^{-\frac{N}{2}-\frac{|\alpha|}{2}}
		\Long{t}^{-\frac{N}{4}-\frac{|\alpha|}{4}}, 
		\quad 
		\|
			\partial_x^\alpha 
			F_j(t)
		\|_{\infty}
		\le 
		C
		t^{-\frac{N}{2j}-\frac{|\alpha|}{2j}},
	\]
	for $j\in\{1, 2\}$ and $t>0$.
\end{Lemma}

\begin{Lemma}
\label{lem:L^1}
	Let 
	$N \ge 1$ 
	and $\alpha\in\mathbb{Z}_+^N$.
	\begin{enumerate}
		\item
			There exists $C=C(N, \alpha)>0$ such that
			\[
				\|
					\partial_x^\alpha F(t)
				\|_1
				\le 
				C
				\Short{t}^{-\frac{|\alpha|}{2}}
				\Long{t}^{-\frac{|\alpha|}{4}}, 
			\]
			for $t>0$.
		\item
			 Assume that $N\ge2$ and $|\alpha|\ge1$.
			Then there exists $C=C(N, \alpha)>0$ such that
			\[
				\|
					\partial_x^{\alpha} 
					F_j(t)
				\|_1
				\le 
				C
				t^{-\frac{|\alpha|}{2j}}
			\]
			for $j \in \{ 1, 2 \}$ and $t>0$.
	\end{enumerate}
\end{Lemma}

In the rest of this section, 
we prove Lemmas~\ref{lem:L^infty} and \ref{lem:L^1}, 
and 
derive several properties of $F(t)\ast f$
as consequences of Proposition~\ref{prop:semigroup}-(1).

%============================================================================%
\subsection{Estimates associated with $h$}
\label{subsec:h}
%============================================================================%

In this subsection, we introduce some estimates for $h$, 
which are required in the proofs of Lemmas~\ref{lem:L^infty} and \ref{lem:L^1}.
We first show two pointwise estimates associated with $h$.

%===================%
\begin{Lemma}
\label{lem:h-1}
	Let $N\ge1$ and $\alpha \in \mathbb{Z}^N_{+}$.
	Then there exists $C=C(N, \alpha)>0$ such that
	\[
		|
			\partial_{\xi}^\alpha 
			h(\xi)
		|
		\le 
		C
		|\xi|^{-|\alpha|}
		h(\xi)
	\]
	for $\xi\in\mathbb{R}^N$.
\end{Lemma}
\begin{proof}
	We observe from the Leibniz rule that
	\[
		\partial_{\xi}^\alpha 
		h(\xi)
		=
		\sum_{
				\ell
				=
				\max
				\{ 0, \frac{ |\alpha|-4 }{2} \}
			}^{|\alpha|}
			\sum_{
					|\beta|
					=
					4
					+
					2
					\ell
					-
					|\alpha|
				}
				\dfrac{
						C(\beta)\xi^\beta
					}{
						(1+|\xi|^2)^{|\alpha|+1}
					}
	\]
	for $\xi\in\mathbb{R}^N$.
	This implies that
	\[
	\begin{split}
		|
			\partial_{\xi}^\alpha 
			h(\xi)
		|
		\le \, &
		C
		\sum_{
				\ell
				=
				\max
				\{
					0, \frac{|\alpha|-4}{2}
				\}
			}^{|\alpha|}
		\dfrac{
				|\xi|^{4+2\ell-|\alpha|}
			}{
				(
					1
					+
					|\xi|^2
				)^{|\alpha|+1}
			}\\
		\le \, & 
		C
		|\xi|^{-|\alpha|}
		\dfrac{
				|\xi|^{4}
			}{
				1
				+|
				\xi|^2
			}\\
			= \, &
			C
			|\xi|^{-|\alpha|}
			h(\xi)
	\end{split}
	\]
	for $\xi\in\mathbb{R}^N$.
	This is the desired estimate of Lemma~\ref{lem:h-1}.
	Thus, we complete the proof of Lemma~\ref{lem:h-1}.
\end{proof}
%===================%

%===================%
\begin{Lemma}
\label{lem:h-2}
	Let $N\ge1$.
	Then it holds that
	\[
			h(\xi)
			\ge
			\left\{ 
				\begin{split}
					&\dfrac{1}{4}
					|\xi|^2
					+
					\dfrac{1}{4}
					&& \text{if} \quad |\xi|\ge1,\\
					&\dfrac{1}{5}
					|\xi|^4
					&& \text{if} \quad |\xi|\le 2.
				\end{split}
			\right.
	\]
\end{Lemma}
\begin{proof}
	If $|\xi| \ge 1$, then we have
	\[
		h(\xi)
		=
		\dfrac{1}{4}
		|\xi|^2
		+
		\dfrac{1}{4}
		\times 
		\dfrac{
					|\xi|^2
				}{
					1+|\xi|^2
				}
		\times 
		(
			3
			|\xi|^2
			-
			1
		)
		\ge 
		\dfrac{1}{4}
		|\xi|^2
		+
		\dfrac{1}{4}
		\times 
		\dfrac{1}{2}
		\times 
		(3-1)
		=
		\dfrac{1}{4}
		|\xi|^2
		+
		\dfrac{1}{4}.
	\]
	On the other hand, 
	if $|\xi|\le 2$, then we have
	\[
		h(\xi)
		=
		\dfrac{
				|\xi|^4
				}{
					1+|\xi|^2
				}
		\ge
		\dfrac{
					|\xi|^4
				}{
					1+2^2
				}
		= 
		\dfrac{1}{5}
		|\xi|^4.
	\]
	Therefore, 
	the proof of Lemma~\ref{lem:h-2} is complete.
\end{proof}
%===================%

We next give pointwise estimates for derivatives of the Fourier transforms of $F_j$ and $F$.

%===================%
\begin{Lemma}
\label{lem:h-3}
	Let $N\ge1$ and $\alpha$, 
	$\beta\in\mathbb{Z}^N_{+}$ with $|\alpha|<|\beta|$.
	Then there exists $C=C(N, \alpha, \beta)>0$ such that 
	\begin{alignat}{1}
	\label{eq:h-3-1}
		|
			\partial_\xi^{\beta}
			\left(
				\xi^\alpha
				\Phi_{j}(\xi)
				\mathrm{e}^{-th(\xi)}
			\right)
		|
		\le \, &
		C 
		|\xi|^{|\alpha|-|\beta|}
		\mathrm{e}^{-th(\xi)}
		\chi_j(\xi)
		\sum_{\ell=0}^{|\beta|}
			t^\ell 
			h(\xi)^\ell,\\
		\label{eq:h-3-2}
		|
			\partial_\xi^{\beta}
			\left(
				\xi^\alpha 
				\mathrm{e}^{-th(\xi)}
			\right)
		|
		\le \, & 
		C 
		|\xi|^{|\alpha|-|\beta|}
		\mathrm{e}^{-th(\xi)}
		\sum_{\ell=1}^{|\beta|}
			t^\ell 
			h(\xi)^\ell,
	\end{alignat}
	for $\xi\in\mathbb{R}^N$, $t>0$, and $j\in\{1, 2\}$, 
	where 
	$
		\chi_1(\xi)
		\coloneqq 
		\chi_{\mathbb{R}^N\setminus B_1(0)}(\xi)
	$ 
	and 
	$
		\chi_2(\xi)
		\coloneqq 
		\chi_{B_2(0)}(\xi)
	$.
\end{Lemma}
\begin{proof}
	Applying the Leibniz rule, we obtain
	\begin{equation}
	\label{eq:lemh-3-1}
		\partial_\xi^{\beta}
		\left(
			\xi^\alpha
			\Phi_{j}(\xi)
			\mathrm{e}^{-th(\xi)}
		\right)
		=
		\sum_{
			\beta_1+\beta_2+\beta_3
			=
			\beta
			}
		C(\beta_1, \beta_2, \beta_3)
		\partial^{\beta_{1}}_{\xi}
		(
			\xi^\alpha
		)
		\times 
		\partial^{\beta_{2}}_{\xi}
		(
			\Phi_j(\xi)
		)
		\times 
		\partial^{\beta_{3}}_{\xi}
		(
			\mathrm{e}^{-th(\xi)}
		)
	\end{equation}
	for $j \in \{ 1, 2 \}$ and $\xi \in \mathbb{R}^N$.
	A direct calculation gives
	\begin{equation}
	\label{eq:lemh-3-2}
	\begin{alignedat}{2}
		|
			\partial^{\beta_{1}}_{\xi}(\xi^\alpha)
		|
		\le \, & 
		C
		|\xi|^{|\alpha|-|\beta_1|}
		&& \qquad \text{if} \quad \alpha \ge \beta_1,\\
		|
			\partial^{\beta_{2}}_{\xi}
			(
				\Phi_{j}(\xi)
			)
		|
		\le \, & 
		C
		|\xi|^{-|\beta_2|}
		\chi_{j}(\xi)
		&&\qquad \text{if} \quad |\beta_2| \neq0,
	\end{alignedat}
	\end{equation}
	for $j\in\{1,2\}$ and $\xi\in\mathbb{R}^N$.
	On the other hand, if $|\beta_3|\neq0$, since 
	\[
		\partial^{\beta_{3}}_{\xi}
		(
			\mathrm{e}^{-th(\xi)}
		)
		=
		\mathrm{e}^{-th(\xi)}
		\sum^{|\beta_3|}_{\ell=1}
			(-t)^{\ell}
			\sum_{
				\substack{
					\gamma_1,\cdots,\gamma_{\ell}\in \mathbb{N}^\ell \\ 
					\gamma_1+\cdots+\gamma_\ell=\beta_3
					}
				} 
				C(\gamma_1, \cdots, \gamma_{\ell})
				\partial_\xi^{\gamma_1}
				h(\xi) 
				\times 
				\cdots 
				\times 
				\partial_\xi^{\gamma_\ell}
				h(\xi)
	\]
	for $\xi \in \mathbb{R}^N$, 
	we observe from Lemma~\ref{lem:h-1} that
	\[
	\begin{split}
		|
			\partial^{\beta_{3}}_{\xi}
			(
				\mathrm{e}^{-th(\xi)}
			)
		|
		\le \, & 
		C
		\mathrm{e}^{-th(\xi)}
		\sum^{|\beta_3|}_{\ell=1}
			t^{\ell}
			\sum_{
				\substack{
					\gamma_1,\cdots, \gamma_{\ell} \in \mathbb{N}^\ell \\ 
					\gamma_1+\cdots+\gamma_\ell=\beta_3
					}
				} 
				|\xi|^{-(|\gamma_1|+\cdots+|\gamma_\ell|)}
				h(\xi)^\ell\\
		= \, & 
		C
		|\xi|^{-|\beta_3|}
		\mathrm{e}^{-th(\xi)}
		\sum^{|\beta_3|}_{\ell=1}
			t^\ell 
			h(\xi)^\ell
	\end{split}
	\]
	for $\xi\in\mathbb{R}^N$.
	This, together with \eqref{eq:lemh-3-1} and \eqref{eq:lemh-3-2}, implies that
	\begin{alignat*}{1}
		|
			\partial_\xi^{\beta}
			\left(
				\xi^\alpha
				\Phi_j(\xi)
				\mathrm{e}^{-th(\xi)}
			\right)
		|
		\le \, &  
		C
		\mathrm{e}^{-th(\xi)}
		\sum_{
			\substack{
				\beta_1+\beta_2=\beta\\
				\beta_1 \le \alpha, \,
				|\beta_2| \neq 0
				}
			}
			|
				\partial^{\beta_{1}}_{\xi}
				(
					\xi^\alpha
				)
			|
			\times
			|
				\partial^{\beta_{2}}_{\xi}
				(
					\Phi_j(\xi)
				)
			|\\
		& \, 
		+
		C
		\Phi_j(\xi)
		\sum_{
			\substack{
				\beta_1+\beta_3=\beta\\ 
				|\beta_3|\ne 0
				}
			}
			|
				\partial^{\beta_{1}}_{\xi}
				(
					\xi^\alpha
				)
			|
			\times
			|
				\partial^{\beta_{3}}_{\xi}
				(
					\mathrm{e}^{-th(\xi)}
				)
			|\\
		& \, +
		C
		\sum_{
			\substack{
				\beta_1+\beta_2+\beta_3=\beta\\
				\beta_1\le \alpha, \,
				|\beta_2| \neq 0, \,
				|\beta_3|\neq 0
				}
			}
			|
				\partial^{\beta_{1}}_{\xi}
				(
					\xi^\alpha
				)
			|
			\times
			|
				\partial^{\beta_{2}}_{\xi}(\Phi_{j}(\xi))
			|
			\times
			|\partial^{\beta_{3}}_{\xi}(\mathrm{e}^{-th(\xi)})|\\
		\le \, & C|\xi|^{|\alpha|-|\beta|}\mathrm{e}^{-th(\xi)}\chi_j(\xi)\sum^{|\beta|}_{\ell=0}t^\ell h(\xi)^\ell
	\end{alignat*}
	for $j\in\{1,2\}$ and $\xi\in\mathbb{R}^N$.
	Thus, \eqref{eq:h-3-1} follows.
	On the other hand, since 
	\[
		\partial_\xi^{\beta}
		\left(
			\xi^\alpha \mathrm{e}^{-th(\xi)}
		\right)
		=
		\sum_{\beta_1+\beta_2=\beta}
		C(\beta_1, \beta_2)
		\partial^{\beta_{1}}_{\xi}
		(
			\xi^\alpha
		)
		\times 
		\partial^{\beta_{2}}_{\xi}
		(
			\mathrm{e}^{-th(\xi)}
		)
	\]
	for $\xi\in\mathbb{R}^N$, 
	we apply the same estimates as in the case of \eqref{eq:h-3-1} to obtain
	\begin{alignat*}{1}
		|
			\partial_\xi^{\beta}
			\left(
				\xi^\alpha 
				\mathrm{e}^{-th(\xi)}
			\right)
		|
		\le \, & 
		C
		\sum_{
			\substack{
				\beta_1+\beta_2=\beta\\ 
				\beta_1\le \alpha, \, |\beta_2|\neq0
				}
				}
			|
				\partial^{\beta_{1}}_{\xi}
				(
					\xi^\alpha
				)
			|
			\times
			|
				\partial^{\beta_{2}}_{\xi}
				(
					\mathrm{e}^{-th(\xi)}
				)
			|\\
		\le \, &  
		C
		|\xi|^{|\alpha|-|\beta|}
		\mathrm{e}^{-th(\xi)}
		\sum^{|\beta|}_{\ell=1}
			t^\ell 
			h(\xi)^\ell
	\end{alignat*}
	for $\xi \in \mathbb{R}^N$.
	Thus \eqref{eq:h-3-2} follows.
	Therefore, the proof of Lemma~\ref{lem:h-3} is complete.
\end{proof}
%===================%

We finally give an estimate for integrals involving $h$.

%===================%
\begin{Lemma}
\label{lem:h-4}
	Let $N\ge1$, $s\in\mathbb{R}$, $\sigma\ge 0$, and $r\in[1,\infty)$ with $N+(s+2\sigma)r>0$.
	Then there exists $C=C(N, s, \sigma, r)>0$ such that
	\[
	\begin{split}
		 \|
		 	|\cdot|^s
			h(\cdot)^\sigma
			\mathrm{e}^{-th(\cdot)}
			\chi_{j}(\cdot)
		\|_r
		\le \, & 
		C
		t^{-\frac{N}{2jr}-\frac{s}{2j}-\sigma},\\
		\|
			|\cdot|^s
			h(\cdot)^\sigma
			\mathrm{e}^{-th(\cdot)}
		\|_r
		\le \, & 
		C
		t^{-\sigma}
		\Short{t}^{-\frac{N}{2r}-\frac{s}{2}}
		\Long{t}^{-\frac{N}{4r}-\frac{s}{4}},
	\end{split}
	\]
	for $j \in \{ 1, 2 \}$ and $t>0$.
\end{Lemma}
\begin{proof}
	By Lemma~\ref{lem:h-2} and 
	\[
	h(\xi)\le
	\left\{
	\begin{alignedat}{2}
		&|\xi|^2 
		&&\quad \text{if}\quad {|\xi|\ge 1}, \\
		&|\xi|^4 
		&&\quad\text{if}\quad {|\xi|<2},
	\end{alignedat}
	\right.
	\]
	we have
	\begin{alignat*}{1}
		\|
			|\cdot|^{s}
			h(\cdot)^\sigma
			\mathrm{e}^{-th(\cdot)}
			\chi_{1}(\cdot)
		\|_r
		= \, &
		\left(
			\int_{\mathbb{R}^N\setminus B_1(0)}
				\left(
					|\xi|^s
					h(\xi)^\sigma
					\mathrm{e}^{-th(\xi)}
				\right)^r
			\, \mathrm{d}\xi
		\right)^{\frac1r}\\
		\le \, &
		\mathrm{e}^{-\frac{1}{4}t}
		\left(
			\int_{\mathbb{R}^N}
				\left(
					|\xi|^{s+2\sigma}
					\mathrm{e}^{-\frac{t}{4}|\xi|^2}
				\right)^r
			\, \mathrm{d}\xi
		\right)^{\frac1r}\\
		= \, &
		t^{-\frac{N}{2r}-\frac{s}{2}-\sigma}
		\mathrm{e}^{-\frac{1}{4}t}
		\left(
			\int_{\mathbb{R}^N}
				\left(
					|\xi|^{s+2\sigma}
					\mathrm{e}^{-\frac{1}{4}|\xi|^2}
				\right)^r
			\, \mathrm{d}\xi
		\right)^{\frac1r}\\
		\le \, &
		C
		t^{-\frac{N}{2r}-\frac{s}{2}-\sigma}
		\mathrm{e}^{-\frac14 t},\\
		\|
			|\cdot|^{s}
			h(\cdot)^\sigma
			\mathrm{e}^{-th(\cdot)}
			\chi_{2}(\cdot)
		\|_r
		= \, & 
		\left(
			\int_{B_2(0)}
				\left(
					|\xi|^{s}
					h(\xi)^{\sigma}
					\mathrm{e}^{-th(\xi)}
				\right)^r
			\, \mathrm{d}\xi
		\right)^{\frac1r}\\
		\le \, & 
		\min
		\left\{
			\int_{B_2(0)}
				|\xi|^{(s+4\sigma)r}
			\, \mathrm{d}\xi, 
			\int_{\mathbb{R}^N}
				\left(
					|\xi|^{s+4\sigma}
					\mathrm{e}^{-\frac{t}{5}|\xi|^4}
				\right)^r
			\, \mathrm{d}\xi
		\right\}^{\frac1r}\\
		\le \, & 
		C
		\min
		\left\{
			1, 
			t^{-\frac{N}{4}-\frac{(s+4\sigma)r}{4}}
			\int_{\mathbb{R}^N}
				\left(
					|\xi|^{s+4\sigma}
					\mathrm{e}^{-\frac{1}{5}|\xi|^4}
				\right)^r
			\, \mathrm{d}\xi
		\right\}^{\frac1r}\\
		\le \, & 
		C
		\min\{
			1, 
			t^{-\frac{N}{4r}-\frac{s}{4}-\sigma}
		\},
	\end{alignat*}
	for $t > 0$.
	In particular, 
	we obtain the desired estimate of 
	$
		\|
			|\cdot|^{s}
			h(\cdot)^\sigma
			\mathrm{e}^{-th(\cdot)}
			\chi_{j}(\cdot)
		\|_r
	$.
	Moreover, 
	since
	\begin{alignat*}{1}
		t^{-\frac{N}{2r}-\frac{s}{2}-\sigma}
		\mathrm{e}^{-\frac{1}{4}t}
		= \, &
		t^{-\sigma}
		\Short{t}^{-\frac{N}{2r}-\frac{s}{2}}
		\Long{t}^{-\frac{N}{4r}-\frac{s}{4}}
		\times 
		\Long{t}^{-\frac{N}{4r}-\frac{s}{4}}
		\mathrm{e}^{-\frac{1}{4}t}\\
		\le \, &
		C
		t^{-\sigma}
		\Short{t}^{-\frac{N}{2r}-\frac{s}{2}}
		\Long{t}^{-\frac{N}{4r}-\frac{s}{4}},\\
		\min
		\{
			1, 
			t^{-\frac{N}{4r}-\frac{s}{4}-\sigma}
		\}
		= \, &
		t^{-\sigma}
		\Short{t}^{-\frac{N}{2r}-\frac{s}{2}}
		\Long{t}^{-\frac{N}{4r}-\frac{s}{4}}
		\times 
		\min
		\{
			t^{\frac{N+(s+2\sigma)r}{2r}}
			\Long{t}^{-\frac{N}{4r}-\frac{s}{4}}, 
			\Short{t}^{\frac{N}{4r}+\frac{s}{4}}
		\}\\
		\le \, & 
		C
		t^{-\sigma}
		\Short{t}^{-\frac{N}{2r}-\frac{s}{2}}
		\Long{t}^{-\frac{N}{4r}-\frac{s}{4}},
	\end{alignat*}
	for $t > 0$, 
	we obtain
	\begin{alignat*}{1}
		\|
			|\cdot|^{s}
			h(\cdot)^\sigma
			\mathrm{e}^{-th(\cdot)}
		\|_r
		\le \, & 
		\|
			|\cdot|^{s}
			h(\cdot)^\sigma
			\mathrm{e}^{-th(\cdot)}
			\chi_{1}(\cdot)
		\|_r
		+
		\|
			|\cdot|^{s}
			h(\cdot)^\sigma
			\mathrm{e}^{-th(\cdot)}
			\chi_{2}(\cdot)
		\|_r\\
		\le \, &
		C
		t^{-\frac{N}{2r}-\frac{s}{2}-\sigma}
		\mathrm{e}^{-\frac{1}{4}t}
		+
		C
		\min
		\{
			1, 
			t^{-\frac{N}{4r}-\frac{s}{4}-\sigma}
		\}\\
		\le \, & 
		C
		t^{-\sigma}
		\Short{t}^{-\frac{N}{2r}-\frac{s}{2}}
		\Long{t}^{-\frac{N}{4r}-\frac{s}{4}}
	\end{alignat*}
	for $t > 0$.
	This is the desired estimate of 
	$
		\|
			|\cdot|^{s}
			h(\cdot)^\sigma
			\mathrm{e}^{-th(\cdot)}
		\|_r
	$, 
	and thus we complete the proof of Lemma~\ref{lem:h-4}.
\end{proof}
%===================%

%============================================================================%
\subsection{Proofs of Lemmas~\ref{lem:L^infty} and \ref{lem:L^1}}
\label{subsec:lem-1}
%============================================================================%

We first give the proof of Lemma~\ref{lem:L^infty}.
\begin{proof}[Proof of Lemma~\ref{lem:L^infty}]
	By the definition of $F_j$ and Lemma~\ref{lem:h-4}, we have
	\begin{alignat*}{1}
		|
			\partial_x^\alpha F(x, t)
		|
		\le \, &
		C
		\|
			|\cdot|^{|\alpha|}
			\mathrm{e}^{-th(\cdot)}
		\|_{1}
		\le 
		C
		\Short{t}^{-\frac{N}{2}-\frac{|\alpha|}{2}}
		\Long{t}^{-\frac{N}{4}-\frac{|\alpha|}{4}},\\
		|
			\partial_x^\alpha F_{j}(x, t)
		|
		\le \, & 
		C
		\|
			|\cdot|^{|\alpha|}
			\mathrm{e}^{-th(\cdot)}
			\chi_j(\cdot)
		\|_1
		\le 
		C 
		t^{-\frac{N}{2j}-\frac{|\alpha|}{2j}},
	\end{alignat*}
	for $j \in \{ 1, 2 \}$, $x \in \mathbb{R}^N$, 
	and $t>0$.
	Therefore, 
	the desired estimate of Lemma~\ref{lem:L^infty} follows.
\end{proof}

The proof of Lemma~\ref{lem:L^1} is based on the following three lemmas:
\begin{Lemma}[{\cite[Lemma~4.5]{CKKW2021}}]
\label{lem:CKKW}
	Let $N \ge 1$ and $k > 0$ satisfy $k > N/2$.
	Then there exists
	$
		C
		=
		C(N, k)
		>
		0
	$ 
	such that
	\[
		\|
			f
		\|_{1}
		\le 
		C
		\|
			f
		\|^{1-\frac{N}{2k}}_2
		\|
			|\cdot|^kf
		\|^{\frac{N}{2k}}_2
	\]
	for $f\in \mathcal{S}$.
\end{Lemma}
This lemma is exactly the same as in Cygan--Karch--Krawczyk--Wakui~\cite[Lemma~4.5]{CKKW2021}.
Hence we omit the proof of Lemma~\ref{lem:CKKW}.
\begin{Lemma}
\label{lem:L^2}
	Let $N\ge1$ and $\alpha\in\mathbb{Z}^N_{+}$.
	Then there exists $C=C(N, \alpha)>0$ such that
	\[
		\|
			\partial_x^\alpha 
			F(t)
		\|_{2}
		\le 
		C
		\Short{t}^{-\frac{N}{4}-\frac{|\alpha|}{2}}
		\Long{t}^{-\frac{N}{8}-\frac{|\alpha|}{4}},
		\quad
		\|
			\partial_x^\alpha 
			F_j(t)
		\|_{2}
		\le 
		C
		t^{-\frac{N}{4j}-\frac{|\alpha|}{2j}},
	\]
	for $j \in \{ 1, 2 \}$ and $t>0$.
\end{Lemma}
\begin{proof}
	Applying the Plancherel theorem, we observe from Lemma~\ref{lem:h-4} that
	\begin{alignat*}{1}
		\|
			\partial_x^\alpha F(t)
		\|_2
		\le \, & 
		C
		\|
			|\cdot|^{|\alpha|}
			\mathrm{e}^{-th(\cdot)}
		\|_{2}
		\le 
		C
		\Short{t}^{-\frac{N}{4}-\frac{|\alpha|}{2}}
		\Long{t}^{-\frac{N}{8}-\frac{|\alpha|}{4}},\\
		\|
			\partial_x^\alpha F_{j}(t)
		\|_2
		\le \, & 
		C
		\|
			|\cdot|^{|\alpha|}
			\mathrm{e}^{-th(\cdot)}
			\chi_j(\cdot)
		\|_2
		\le 
		C 
		t^{-\frac{N}{4j}-\frac{|\alpha|}{2j}},
	\end{alignat*}
	for $j \in \{ 1, 2 \}$ and $t > 0$.
	Therefore, 
	the desired estimate of Lemma~\ref{lem:L^2} follows.
\end{proof}

\begin{Lemma}
\label{lem:L^2derivative}
	Let $N \ge 1$, $k > 0$, $m \in \mathbb{Z}_{+}$, 
	and 
	$
		\alpha
		\in
		\mathbb{Z}^N_{+}
	$ 
	satisfy 
	$
		0
		<
		k
		<
		2
		m
		<
		N/2
		+
		k
	$.
	\begin{enumerate}
		\item
			Assume that $k<N/2+|\alpha|+2$ 
			and
			$
				0
				\le 
				|\alpha|
				<
				2m
			$.
			Then there exists $C=C(N, k, m, \alpha)>0$ such that
			\[
				\|
					|\cdot|^k 
					\partial_x^\alpha F(t)
				\|_2
				\le 
				C
				\Short{t}^{-\frac{N}{4}-\frac{|\alpha|}{2}+\frac{k}{2}}
				\Long{t}^{-\frac{N}{8}-\frac{|\alpha|}{4}+\frac{k}{4}}
			\]
			for $t > 0$.
		\item
			Assume that $k < N/2 + |\alpha|$.
			Then there exists $C=C(N, k, m, \alpha)>0$ such that
			\[
				\|
					|\cdot|^k
					\partial_x^\alpha F_{j}(t)
				\|_2
				\le 
				C
				t^{-\frac{N}{4j}-\frac{|\alpha|}{2j}+\frac{k}{2j}}
			\]
			for $j \in \{ 1, 2 \}$ and $t > 0$.
	\end{enumerate}
\end{Lemma}
\begin{proof}
	Note that $0<2m-k<N$. 
	For $f \in \mathcal{S}$, 
	it holds that 
	$
		|\cdot|^k f
		\in 
		L^1(\mathbb{R}^N)
	$ 
	and its Fourier transform $\mathcal{F}[|\cdot|^k f]$ is given by 
	$
		C_{N, m, k}
		[
			|\cdot|^{2m-k-N} \ast (-\Delta)^m \mathcal{F}[f]
		]
	$.
	Applying the Hardy--Littlewood--Sobolev inequality (see e.g. Grafakos~\cite[Theorem~1.2.3]{G2014book-2}), 
	we have
	\[
		\|
			[
				|\cdot|^{2m-k-N} \ast (-\Delta)^m \mathcal{F}[f]
			]
		\|_2
		\le 
		C
		\|
			(-\Delta)^m\mathcal{F}[f]
		\|_r
	\]
	for $f \in \mathcal{S}$, 
	where 
	$
		r
		\coloneqq 
		2N/(N+2(2m-k))
		\in
		(1, 2)
	$.
	Combining this with the Plancherel theorem, we obtain
	\begin{equation}
	\label{eq:L^2-1}
		\|
			|\cdot|^k f
		\|_2
		\le 
		C
		\|
			(-\Delta)^m\mathcal{F}[f]
		\|_r
	\end{equation}
	for $f \in \mathcal{S}$.
	
	We show assertion~(1).
	Applying \eqref{eq:L^2-1} with 
	$
		f
		=
		\partial_x^\alpha F(\cdot, t)
	$, 
	we have
	\begin{equation}
	\label{eq:L^2-2}
		\|
			|\cdot|^k 
			\partial^\alpha_x F(t)
		\|_2
		\le 
		C
		\|
			(-\Delta)^m\mathcal{F}[\partial^\alpha_x F(t)]
		\|_r
	\end{equation}
	for $t > 0$.
	Moreover, 
	since Lemma~\ref{lem:h-3} implies that
	\[
		\left|
			(-\Delta)_{\xi}^m\mathcal{F}[\partial^\alpha_x F(t)](\xi)
		\right|
		= 
		\left|
			(-\Delta)_\xi^m
			\left(
				\xi^\alpha 
				\mathrm{e}^{-th(\xi)}
			\right)
		\right|
		\le 
		C
		|\xi|^{|\alpha|-2m}
		\mathrm{e}^{-th(\xi)}
		\sum^{2m}_{\ell=1}t^\ell 
			h(\xi)^\ell
	\]
	for $t > 0$, 
	we observe from \eqref{eq:L^2-2} that
	\[
		\|
			|\cdot|^k 
			\partial^\alpha_x F(t)
		\|_2
		\le 
		C
		\sum^{2m}_{\ell=1}
			t^\ell
			\left\|
				|\cdot|^{|\alpha|-2m}
				h(\cdot)^\ell 
				\mathrm{e}^{-th(\cdot)}
			\right\|_r
	\]
	for $t>0$.
	Since
	\[
		N+(|\alpha|-2m+2\ell)r\ge N+\dfrac{2N(|\alpha|-2m+2)}{N+2(2m-k)}=\dfrac{2N}{N+2(2m-k)}\times \left(\dfrac{N}{2}+|\alpha|+2-k\right)>0,
	\]
	it follows from Lemma~\ref{lem:h-4} that
	\[
		\|
			|\cdot|^k 
			\partial^\alpha_x 
			F(t)
		\|_2
		\le 
		C
		\sum^{2m}_{\ell=1}
			t^\ell
			\times 
			t^{-\ell}
			\Short{t}^{-\frac{N}{2r}-\frac{|\alpha|-2m}{2}}
			\Long{t}^{-\frac{N}{4r}-\frac{|\alpha|-2m}{4}}
		=
		C
		\Short{t}^{-\frac{N}{4}-\frac{|\alpha|}{2}+\frac{k}{2}}
		\Long{t}^{-\frac{N}{8}-\frac{|\alpha|}{4}+\frac{k}{4}}
	\]
	for $t>0$.
	Therefore, assertion~(1) follows.
	
	We finally show assertion~(2).
	By applying an argument similar to that used for assertion~(1), one can verify that
	\[
		\|
			|\cdot|^k 
			\partial^\alpha_x 
			F_j(t)
		\|_{2}
		\le 
		C
		\sum^{2m}_{\ell=0}
		t^\ell
		\left\|
			|\cdot|^{|\alpha|-2m}
			h(\cdot)^\ell 
			\mathrm{e}^{-th(\cdot)}
			\chi_j(\cdot)
		\right\|_r
	\]
	for $j \in \{ 1, 2 \}$ and $t > 0$.
	Since
	\[
		N
		+
		(
			|\alpha|-2m+2\ell
		)
		r
		\ge 
		N
		+
		\dfrac{2N(|\alpha|-2m)}{N+2(2m-k)}
		=
		\dfrac{2N}{N+2(2m-k)}
		\times 
		\left(
			\dfrac{N}{2}
			+
			|\alpha|-k
		\right)
		>
		0,
	\]
	it follows from Lemma~\ref{lem:h-4} that
	\[
		\|
			|\cdot|^k 
			\partial^\alpha_x F_j(t)
		\|_2
		\le 
		C
		\sum^{2m}_{\ell=0}
			t^\ell
			\times 
			t^{-\frac{N}{2jr}-\frac{|\alpha|-2m}{2j}-\ell}
		=
		C
		t^{-\frac{N}{4j}-\frac{|\alpha|}{2j}+\frac{k}{2j}}
	\]
	for $j \in \{ 1, 2 \}$ and $t > 0$.
	Therefore, assertion~(2) follows and the proof of Lemma~\ref{lem:L^2derivative} is complete.
\end{proof}
We are now in a position to complete the proof of Lemma~\ref{lem:L^1}.
\begin{proof}[Proof of Lemma~\ref{lem:L^1}]
	We first show assertion~(1).
	Set
	\[
		k
		\coloneqq
		2
		\left\lfloor
			\dfrac{N}{4}
		\right\rfloor
		+
		2
		\left\lfloor
			\dfrac{|\alpha|}{2}
		\right\rfloor
		+
		\dfrac{7}{4},
		\quad
		m
		\coloneqq
		\left\lfloor
			\dfrac{N}{4}
		\right\rfloor
		+
		\left\lfloor
			\dfrac{|\alpha|}{2}
		\right\rfloor
		+
		1.
	\]
	Then we observe that
	\[
		\dfrac{N}{2}
		<
		k
		<
		2m
		<
		\dfrac{N}{2}
		+
		k,
		\quad
		|\alpha|<2m,
		\quad \text{and} \quad 
		k
		<
		\dfrac{N}{2}+|\alpha|+2.
	\]
	Therefore, we can apply Lemmas~\ref{lem:CKKW}, \ref{lem:L^2}, 
	and \ref{lem:L^2derivative} to derive
	\begin{alignat*}{1}
		\| 
			\partial_x^\alpha F(t)
		\|_1
		\le \, & 
		C
		\|
			\partial_x^\alpha F(t)
		\|^{1-\frac{N}{2k}}_2
		\|
			|\cdot|^k 
			\partial_x^\alpha F(t)
		\|^{\frac{N}{2k}}_2\\
		\le \, & 
		C
		\left(
			\Short{t}^{-\frac{N}{4}-\frac{|\alpha|}{2}}
			\Long{t}^{-\frac{N}{8}-\frac{|\alpha|}{4}}
		\right)^{1-\frac{N}{2k}}
		\left(
			\Short{t}^{-\frac{N}{4}-\frac{|\alpha|}{2}+\frac{k}{2}}
			\Long{t}^{-\frac{N}{8}-\frac{|\alpha|}{4}+\frac{k}{4}}
		\right)^{\frac{N}{2k}}\\
		=\,&
		C
		\Short{t}^{-\frac{|\alpha|}{2}}
		\Long{t}^{-\frac{|\alpha|}{4}}
	\end{alignat*}
	for $t>0$.
	This is the desired estimate of assertion~(1).
	
	We show assertion~(2).
	Set
	\[
		k
		\coloneqq
		\dfrac{N}{2}
		+
		2
		\left\lfloor
			\dfrac{|\alpha|-1}{2}
		\right\rfloor
		+
		\dfrac{1}{4},
		\quad
		m
		\coloneqq
		\left\lfloor
			\dfrac{N}{4}
		\right\rfloor
		+
		\left\lfloor
			\dfrac{|\alpha|-1}{2}
		\right\rfloor
		+
		1.
	\]
	Then we observe that
	\[
		\dfrac{N}{2}
		<
		k
		<
		2m
		<
		\dfrac{N}{2}
		+
		k
		\quad \text{and} \quad 
		k
		<
		\dfrac{N}{2}+|\alpha|.
	\]
	Therefore, we can apply Lemmas~\ref{lem:CKKW}, \ref{lem:L^2}, and \ref{lem:L^2derivative} to derive
	\begin{alignat*}{1}
		\|\partial^\alpha_xF_j(t)\|_{1}
		\le \, &
		C
		\|
			\partial^\alpha_x
			F_j(t)
		\|^{1-\frac{N}{2k}}_2
		\|
			|\cdot|^k 
			\partial^\alpha_x
			F_j(t)
		\|^{\frac{N}{2k}}_2\\
		\le \, &
		C
		\left(
			t^{-\frac{N}{4j}-\frac{|\alpha|}{2j}}
		\right)^{1-\frac{N}{2k}}
		\left(
			t^{-\frac{N}{4j}-\frac{|\alpha|}{2j}+\frac{k}{2j}}
		\right)^{\frac{N}{2k}}
		=
		C
		t^{-\frac{|\alpha|}{2j}}
	\end{alignat*}
	for $j \in \{ 1, 2 \}$ and $t > 0$.
	This is the desired estimate of assertion~(2).
	Thus, we complete the proof of Lemma~\ref{lem:L^1}.
\end{proof}
%============================================================================%
\subsection{Consequences of Proposition~\ref{prop:semigroup}-(1)}
\label{subsec:consequence}
%============================================================================%
In this subsection, we show several properties associated with the convolution with $F$, which are well-known in the case of the heat kernel:
\begin{Corollary}
\label{cor:semigroup}
	Let $N\ge1$, $1\le p\le q\le \infty$, and $f\in\mathcal{L}^p$. 
	\begin{enumerate}
		\item
			The map $[0, \infty)\ni t\mapsto F(t)\ast f \in \mathcal{L}^p$ belongs to $C([0, \infty); \mathcal{L}^p)$,
			where $F(0)\ast f\coloneqq f$.
		\item
			For $\alpha\in\mathbb{Z}^N_{+}$, the map $(0, \infty)\ni t\mapsto \partial_x^\alpha [F(t)\ast f] \in \mathcal{L}^q$ belongs to $C((0, \infty); \mathcal{L}^q)$.
		\item
			Let $\alpha \in \mathbb{Z}_{+}^N$. If $p<q$, then
			\[
				\lim_{t \searrow 0}
					t^{\frac{N}{2}(\frac{1}{p}-\frac{1}{q})+\frac{|\alpha|}{2}}
					\|
						\partial_x^\alpha [F(t) \ast f]
					\|_q
				=
				0.
			\]
	\end{enumerate}
\end{Corollary}
\begin{proof}
	Set $g(t)\coloneqq F(t)\ast f$ and assume $f\in \mathcal{S}$.
	Since $F(t)\in\mathcal{S}$ for $t>0$, we have $g(t)\in \mathcal{L}^p$ for $1\le p\le \infty$ and $t\ge 0$.
	
	We prove that $g\in C([0, \infty); \mathcal{L}^p)$. Let $t_1>t_2\ge 0$. Simple calculations give
	\begin{alignat*}{1}
		\|g(t_1)-g(t_2)\|_{\infty}
		&\le C\|\mathcal{F}[g(t_1)-g(t_2)]\|_{1}\\
		&\le C\int_{\mathbb{R}^N}|{\rm e}^{-t_1h(\xi)}-{\rm e}^{-t_2h(\xi)}||\mathcal{F}[f](\xi)|\,\mathrm{d}\xi,
	\end{alignat*}
	where $\mathcal{F}[h]$ denotes the Fourier transform of $h$.
	Applying the dominated convergence theorem, 
	we see that $g\in C([0, \infty); \mathcal{L}^\infty)$.
	On the other hand, let $k$ and $m$ be 
	\begin{equation}\label{ex;km}
		k
		\coloneqq
		2
		\left\lfloor
			\dfrac{N}{4}
		\right\rfloor
		+
		\dfrac{7}{4},
		\quad
		m
		\coloneqq
		\left\lfloor
			\dfrac{N}{4}
		\right\rfloor
		+
		1.
	\end{equation}
	We then deduce from Lemma \ref{lem:CKKW}
	and the Plancherel theorem
	that
	\begin{equation}\label{est;L^2t1t2}
	\begin{split}
	\|g(t_1)-g(t_2)\|_{1}
	\le &C\|g(t_1)-g(t_2)\|_{2}^{1-\frac{N}{2k}}\||\cdot|^{k}g(t_1)-|\cdot|^{k}g(t_2)\|_{2}^{\frac{N}{2k}} \\
	\le &C\|({\rm e}^{-t_2h(\xi)}-{\rm e}^{-t_1h(\xi)})\mathcal{F}[f]\|_{2}^{1-\frac{N}{2k}}\||\cdot|^{k}g(t_1)-|\cdot|^{k}g(t_2)\|_{2}^{\frac{N}{2k}}. 
	\end{split}
	\end{equation}
	By Lemmas~\ref{lem:L^1}, \ref{lem:L^2derivative}, and the H\"{o}lder inequality,
	we have 
	\[
	\||\cdot|^kg(t_1)-|\cdot|^kg(t_2)\|_2 
	\le C\sum_{j=1}^2\left[\Short{t_j}^{-\frac{N}{4}+\frac{k}{2}}\Long{t_j}^{-\frac{N}{8}+\frac{k}{4}}\|f\|_1+\||\cdot|^kf\|_2\right].
	\]
	This together with \eqref{est;L^2t1t2} yields that 
	\begin{align*}
	&
	\|g(t_1)-g(t_2)\|_{1} \\
	\le \, & 
	C\|({\rm e}^{-t_2h(\xi)}-{\rm e}^{-t_1h(\xi)})\mathcal{F}[f]\|_{2}^{1-\frac{N}{2k}}
	\left(\sum_{j=1}^2\left[\Short{t_j}^{-\frac{N}{4}+\frac{k}{2}}\Long{t_j}^{-\frac{N}{8}+\frac{k}{4}}\|f\|_1+\||\cdot|^kf\|_2\right]\right)^{\frac{N}{2k}}. 
	\end{align*}
	Therefore, making use of the dominated convergence theorem again,
	we obtain $g\in C([0, \infty); \mathcal{L}^1)$ and, by the H\"{o}lder inequality we have $g\in C([0, \infty); \mathcal{L}^r)$ for $r\in [1,\infty]$ if $f\in \mathcal{S}$.
	Consequently, assertion~(1) follows from \eqref{lqlp_est;F} and the density of $\mathcal{S}$ in $\mathcal{L}^p$.

	We next claim that $\partial_x^\alpha g\in C((0, \infty); \mathcal{L}^q)$ for $p\le q\le \infty$ and $\alpha\in\mathbb{Z}^N_{+}$.
	Fix $\delta>0$ arbitrarily.
	Then we have 
	$
		\partial_x^\alpha g(t)
		=
		F(t-\delta) \ast (\partial_x^\alpha [F(\delta)\ast f])
	$
	for 
	$
		t
		>
		\delta
	$.
	Since $\partial_x^\alpha [F(\delta)\ast f]\in \mathcal{L}^q$ due to Proposition~\ref{prop:semigroup}-(1), assertion~(1) implies that $\partial_x^\alpha g\in C([\delta, \infty); \mathcal{L}^q)$, which gives assertion~(2).

	By Proposition~\ref{prop:semigroup}-(1), we see that
	assertion~(3) holds 
	by combining \eqref{lqlp_est;F} with the same argument as in Brezis--Cazenave~\cite[Lemma~8]{BC1996}.
	As a consequence, we complete the proof of Corollary~\ref{cor:semigroup}.
\end{proof}

%============================================================================%
%============================================================================%
\section{{Global existence of solutions to problem~\eqref{eq:KS_A}}}
\label{sec:perturbation}
%============================================================================%
%============================================================================%
In this section, we investigate the existence and the asymptotic behavior of global-in-time solutions to problem~\eqref{eq:KS_A}.
Similarly to problem~\eqref{eq:KS}, we define the solution to problem~\eqref{eq:KS_A} as follows:
\begin{Definition}
\label{def:KSAsol}
	Let $N\ge3$, $T\in(0, \infty]$, and $v_0\in L^\frac{N}{2}(\mathbb{R}^N)$.
	We say that $v$ is a solution to problem~\eqref{eq:KS_A} if 
	$
		v
		\in 
		C([0, T); L^\frac{N}{2}(\mathbb{R}^N))
	$ 
	with 
	$
		v(0)
		=
		v_0
	$ 
	and $v$ satisfies
	\begin{equation}
	\label{eq:IE}
		v(t)
		=
		F(t) \ast v_0
		-
		\int^t_0
			\nabla \cdot 
			\left[
				F(t-s) \ast 
				(
					v(s) \nabla K \ast v(s)
				)
			\right]
		\, \mathrm{d}s
	\end{equation}
	in $L^\frac{N}{2}(\mathbb{R}^N)$ for $t \in (0, T)$.
	If $T=\infty$, $v$ is called a global-in-time solution to problem~\eqref{eq:KS_A}.
\end{Definition}
The aim of this section is to show the following:
\begin{Theorem}
\label{thm:main}
	Let $N\ge3$.
	Then there exist constants $\varepsilon_* > 0$ and $C>0$ with the following property:
	if 
	$
		v_0 
		\in 
		L^\frac{N}{2}(\mathbb{R}^N)
	$ 
	satisfies
	\begin{equation}
	\label{eq:initial}
		\|
			v_0
		\|_{\frac{N}{2}}
		\le 
		\varepsilon_*,
	\end{equation}
	then there exists a global-in-time solution $v$ to problem~\eqref{eq:KS_A} such that 
	\[
		v, \nabla v\in C((0, \infty); \mathcal{L}^{r})
	\]
	and
	\[
		\Short{t}^{\frac{N}{2}(\frac{2}{N}-\frac{1}{r})}
		\Long{t}^{\frac{N}{4}(\frac{2}{N}-\frac{1}{r})}
		\|
			v(t)
		\|_{r}
		+
		\Short{t}^{\frac{N}{2}(\frac{2}{N}-\frac{1}{r})+\frac{1}{2}}
		\Long{t}^{\frac{N}{4}(\frac{2}{N}-\frac{1}{r})+\frac{1}{4}}
		\|
			\nabla v(t)
		\|_{r}
		\le
		C
		\|
			v_0
		\|_{\frac{N}{2}}
	\]
	for $r \in [N/2, \infty]$ and $t>0$.
	In particular, 
	$v(t)$ converges to $0$ as $t \to \infty$ in $W^{1,r}(\mathbb{R}^N)$
	for $r \in (N/2, \infty]$.
\end{Theorem}
Theorem~\ref{thm:main_KS} is a direct consequence of Theorem~\ref{thm:main}.
Since this is a standard approximation argument, 
we only outline the proof.
Let $u_0$ satisfy \eqref{eq:main_KS_asm} and set $v_0\coloneqq u_0-1$.
Since $v_0$ satisfies \eqref{eq:initial}, 
there exists a global-in-time solution $v$ to problem~\eqref{eq:KS_A} and 
$
	u
	\coloneqq 
	v+1
$
satisfies the desired estimate of Theorem~\ref{thm:main_KS}.
Thus, 
it suffices to show that $u$ is a global-in-time solution to problem~\eqref{eq:KS} with initial data $u_0$.
Fix $(x,t) \in\mathbb{R}^N \times (0, \infty)$ arbitrarily and set 
$
	\zeta(y,s)
	=
	\zeta_{x,t}(y,s)
	\coloneqq 
	G(x-y, t-s)
$, 
where $G$ is the heat kernel defined by \eqref{eq:kernel}.
Multiplying both sides of \eqref{eq:IE} by 
$
	-
	\partial_s\zeta(s)
	-
	\Delta \zeta(s)
	+
	\Delta K \ast \zeta(s)
$, 
integrating over $\mathbb{R}^N \times(0, t)$, 
and integrating by parts, 
we have
\begin{alignat*}{1}
	&
	\int^t_{0}
		\int_{\mathbb{R}^N}
			v(s)
			\left(
				-
				\partial_s \zeta(s)
				-
				\Delta \zeta(s)
				+
				\Delta K \ast \zeta(s)
			\right)
			\, \mathrm{d}y \, \mathrm{d}s\\
	= \, &
	-
	\int_{\mathbb{R}^N}
		v_0
		\left[
			\zeta(t)\ast F(t)
		\right]
		\, \mathrm{d}y
		+
		\int_{\mathbb{R}^N}
			v_0
			\zeta(0)
		\, \mathrm{d}y
		-
		\int^t_{0}
			\int_{\mathbb{R}^N}
				\left(
					v(\tau) 
					\nabla K \ast v(\tau)
				\right)
				\cdot 
				\nabla \zeta(\tau)
			\, \mathrm{d}y
		\, \mathrm{d}\tau\\
	\, & 
	+
	\int^t_{0}
		\int_{\mathbb{R}^N}
			\left(
				v(\tau) \nabla K \ast v(\tau)
			\right)
		\cdot 
		\left[
			\nabla \zeta(t) \ast F(t-\tau)
		\right]
		\, \mathrm{d}y
	\, \mathrm{d}\tau.
\end{alignat*}
Then, by the definition of $\zeta$ we obtain
\begin{alignat*}{1}
	&
	[
		F(t)\ast v_0
	](x)
	-
	\int^t_{0}
		\nabla \cdot
		\left[
			F(t-\tau) \ast 
			\left(
				v(\tau) \nabla K \ast v(\tau)
			\right) 
		\right](x)
	\, \mathrm{d}\tau
	+
	1\\
	= \, &
	[
		G(t)\ast u_0
	](x)
	-
	\int^t_{0}
		\nabla \cdot
		\left[
			G(t-\tau) \ast 
			\left(
				u(\tau) \nabla K\ast u(\tau)
			\right) 
		\right](x)
	\, \mathrm{d}\tau.
\end{alignat*}
Since the left-hand side of the above equation is nothing but $u$, 
it follows that $u$ is a global-in-time solution to problem~\eqref{eq:KS} with initial data $u_0$, 
and thus we complete the proof of Theorem~\ref{thm:main_KS}.

In the rest of this section, we give the proof of Theorem~\ref{thm:main}.
We recall some properties of Bessel potential $K$
defined by \eqref{eq:Bessel}.
It is known that $K\ast f$ satisfies
\begin{equation}
\label{eq:poisson}
	(-\Delta +1)K\ast f=f
\end{equation}
for suitable functions $f$.
Moreover, since 
\[
	K\in L^1(\mathbb{R}^N),
	\quad
	\nabla K
	\in
	\left\{
	\begin{alignedat}{2}
		&L^1(\mathbb{R}^N) \cap L^\infty(\mathbb{R}^N)
		&&\quad\text{if} \quad N=1,\\
		&L^1(\mathbb{R}^N) \cap L_{\rm w}^{\frac{N}{N-1}}(\mathbb{R}^N)
		&&\quad\text{if} \quad N \ge 2,
	\end{alignedat}
	\right.
\]
where $L^r_{\rm w}(\mathbb{R}^N)$ is the weak $L^r$ space, we deduce from the refined Young inequality (see e.g. Grafakos~\cite[Theorem~1.4.25]{G2014book-1}) that:
\begin{Proposition}
\label{prop:HLS}
	Let $N\ge1$, $1\le p\le \infty$, and $f \in \mathcal{L}^p$.
	Then there exists $C=C(N, p)>0$ such that
	\[
		\|
			K\ast f
		\|_{p}
		+
		\|
			\nabla K\ast f
		\|_{p}
		\le 
		C
		\|
			f
		\|_{p}.
	\]
	In addition, if $1<p<N$, then
	\[
		\|
			\nabla K \ast f
		\|_{(\frac{1}{p}-\frac{1}{N})^{-1}}
		\le 
		C
		\|
			f
		\|_p.
	\]
\end{Proposition}
Moreover, it is also known that the Calder\'{o}n--Zygmund type estimate holds (see e.g. Stein~\cite[Chapter~V, Section~3, Theorem~3]{S1970book} and \cite[Chapter~V, Section~6, {\bf 6.6}]{S1970book}):
\begin{Proposition}
\label{prop:CZ}
	Let $N\ge1$ and $1\le p\le \infty$.
	Then there exists $C=C(N, p)>0$ such that
	\[
		\|\Delta K\ast f\|_{p}\le C\|f\|_{p}
	\]
	holds for $f\in \mathcal{L}^p$.
\end{Proposition}
We also recall the classical Sobolev inequality (see e.g. Adams--Fournier~\cite[Theorem~4.31]{AF2003book}):
\begin{Proposition}
\label{prop:Sobolev}
	Let $N\ge2$ and $1\le p<N$.
	Then there exists $C=C(N,p)>0$ such that
	\[
		\|f\|_{(\frac{1}{p}-\frac{1}{N})^{-1}}\le C\|\nabla f\|_{p}
	\]
	holds for $f\in \mathcal{S}$.
\end{Proposition}

%============================================================================%
\subsection{Bilinear estimate}
\label{subsec:bilinear}
%============================================================================%
In this subsection, we introduce a bilinear estimate associated with problem~\eqref{eq:KS_A}.
Set
\[
	X_{a, T}
	\coloneqq
	\left\{
		f
		\colon
		\mathbb{R}^N\times(0, T)\to \mathbb{R}
		\, \middle| \, 
		f, \nabla f\in C((0, T); \mathcal{L}^a)
		\ \text{and}\ 
		\norm{f}_{a,T}
		<
		\infty
	\right\}
\]
for $a\in[N/2, \infty]$ and $T\in(0, \infty]$, where
\[
	\norm{f}_{a, t}
	\coloneqq 
	\sum_{|\alpha|\le 1}
	\sup_{s\in(0, t)}
	\Short{s}^{\frac{N}{2}(\frac{2}{N}-\frac{1}{a})+\frac{|\alpha|}{2}}
	\Long{s}^{\frac{N}{4}(\frac{2}{N}-\frac{1}{a})+\frac{|\alpha|}{4}}
	\|
		\partial_x^\alpha f(s)
	\|_a.
\]
We also define two bilinear terms as
\begin{alignat*}{1}
	&
	\Theta_1[f,g](t)
	\coloneqq
	\int^t_{\frac{t}{2}}\nabla \cdot[F(t-s)\ast\left(f(s)\nabla K\ast g(s)\right)]\,\mathrm{d}s,\\
	&
	\Theta_2[f,g](t)
	\coloneqq
	\int^t_{0}\nabla \cdot[F(t-s)\ast\left(f(s)\nabla K\ast g(s)\right)]\,\mathrm{d}s.
\end{alignat*}
The proof of Theorem~\ref{thm:main} is based on the following estimates:

\begin{Proposition}
\label{prop:bilinear}
	Let $N\ge3$ and $a, b, c\in[N/2, \infty]$ satisfy
	\begin{equation}
	\label{eq:condition}
	\tag{C-1}
		\dfrac{1}{N}+\dfrac{1}{c}\le \dfrac{1}{a}+\dfrac{1}{b}<\dfrac{2}{N}+\dfrac{1}{c},
		\quad
		0\le \dfrac{1}{a}\le \dfrac{2}{N},
		\quad
		\dfrac{1}{N}<\dfrac{1}{b}\le \dfrac{2}{N}.
	\end{equation}
	Then there exists $C=C(N, a, b, c)>0$ such that
	$\Theta_1[f,g]\in X_{c, T}$ with 
	\[
		\norm{\Theta_1[f,g]}_{c, T}
		\le C
		\left(\norm{f}_{a, T}+\norm{f}_{b, T}\right)
		\left(\norm{g}_{a, T}+\norm{g}_{b, T}\right)
	\]
	for $T\in(0, \infty]$ and $f, g\in X_{a,T}\cap X_{b,T}$.
	In addition, if
	\begin{equation}
	\tag{C-2}
	\label{eq:additional}
		\dfrac{1}{a}+\dfrac{1}{b}>\dfrac{3}{N},
	\end{equation}
	then there exists $C=C(N,a,b,c)>0$ such that
	$\Theta_2[f,g]\in X_{c, T}$ with 
	\[
		\norm{\Theta_2[f,g]}_{c, T}
		\le C
		\left(\norm{f}_{a, T}+\norm{f}_{b, T}\right)
		\left(\norm{g}_{a, T}+\norm{g}_{b, T}\right)
	\]
	for $T\in(0, \infty]$ and $f, g\in X_{a,T}\cap X_{b,T}$.
\end{Proposition}

\begin{figure}[H]
\centering
\tikzset{
	dashed/.style={
	dash pattern=on 2pt off 1pt
	}
} 
\begin{subfigure}{0.45\linewidth}
\centering
\begin{tikzpicture}
	\fill[fill=cyan!30] (2,4) -- (2,8) -- (4,8) -- cycle;
	\fill[fill=cyan!80] (2.25,4.5) -- (2.25,5) -- (2,5) -- (2,4.5) -- cycle;
	\fill[fill=cyan!80] (2.25,6.25) -- (2.25,6) -- (2,6) -- (2,6.25) -- cycle;
	\draw[->,>=stealth,thick] (-1.25,0)--(5,0) 
	node [above, font=\small]{$\frac{1}{a}$};
	\draw[->,>=stealth, dashed,thick] (-1,0)--(-1,8.5) 
	node [right, font=\small]{$\frac{1}{c}$};
	\draw[thick, dashed, domain=-1:2] plot(\x, 4);
	\draw[thick, dashed, domain=-1:2] plot(\x, 8);
	\draw (-1,0) node [below, font=\tiny]{$\frac{3}{4N}$};
	\draw (-1.25,0) node [left, font=\tiny]{$0$};
	\draw (-1,4) node [left, font=\tiny]{$\frac{1}{N}$};
	\draw (-1,6) node [left, font=\tiny]{$\frac{3}{2N}$};
	\draw (-1,6.25) -- (-0.5,6.75) node[above, font=\tiny] {$\frac{25}{16N}$};
	\draw (-1,8) node [left, font=\tiny]{$\frac{2}{N}$};
	\draw (-1,4.5) node [left, font=\tiny]{$\frac{9}{8N}$};
	\draw (-1,5) node [left, font=\tiny]{$\frac{5}{4N}$};
	\draw (2,0) node [below, font=\tiny]{$\frac{3}{2N}$};
	\draw (2.25,0) -- (2.75,0.25) node[right, font=\tiny] {$\frac{25}{16N}$};
	\draw (4,0) node [below, font=\tiny]{$\frac{2}{N}$};
	\draw[thick, dashed, domain=0:8.25] plot (2,\x);
	\draw[thick, dashed, domain=0:8.25] plot (2.25,\x);
	\draw[thick, dashed, domain=0:8.25] plot (4,\x);
	\draw[thick, dashed, domain=-1:4.5] plot (\x,4.5);
	\draw[thick, dashed, domain=-1:4.5] plot (\x,5);
	\draw[thick, dashed, domain=-1:4.5] plot (\x,6);
	\draw[thick, dashed, domain=-1:4.5] plot (\x,6.25);
	\draw[thick, domain=-1.25:2.2] plot({\x}, {2*\x+4})
	node[above, font=\tiny]{$\frac{1}{c}=\frac{2}{a}-\frac{1}{N}$};
	\draw[thick, domain=-0.25:4.2] plot({\x}, {2*\x})
	node[above, font=\tiny]{$\frac{1}{c}=\frac{2}{a}-\frac{2}{N}$};
	\draw[ultra thick, color=white, domain=4:8] plot (2,\x);
	\draw[ultra thick, dashed, domain=4:8] plot (2,\x);
	\draw[ultra thick, color=white, domain=2:4] plot ({\x},{2*\x});
	\draw[ultra thick, dashed, domain=2:4] plot ({\x},{2*\x});
	\draw[ultra thick, domain=2:4] plot (\x,8);
	\draw[line width=1pt, fill=white] (4,8) circle (2pt);
	\draw[line width=1pt, fill=black] (2,8) circle (2pt);
	\draw[line width=1pt, fill=white] (2,4) circle (2pt);
	\draw[thick] (2.125,4.75) -- (2.625,3.75) node[draw, fill=white, inner sep=2pt, below, font=\tiny] {Range of $(1/q, 1/q_1)$};
	\draw[thick] (2.125,6.125) -- (2.875,5.625) node[draw, fill=white, inner sep=2pt, below, font=\tiny] {Range of $(1/q, 1/q)$};
\end{tikzpicture}
\caption{\footnotesize Range of $(1/a,1/c)$ where \eqref{eq:condition}, \eqref{eq:additional}, and $a=b$ hold}
\end{subfigure}
\hfill
\begin{subfigure}{0.45\linewidth}
\centering
\begin{tikzpicture}
	\fill[fill=cyan!30] (0,0) -- (0,4) -- (2,8) -- (4,8) -- cycle;
	\fill[fill=cyan!80] (0.5,2) -- (0.5,3) -- (1,3) -- (1,2) -- cycle;
	\draw[->,>=stealth,thick] (-1.25,0)--(5,0) 
	node [above, font=\small]{$\frac{1}{a}$};
	\draw[->,>=stealth, dashed,thick] (-1,0)--(-1,8.5) 
	node [right, font=\small]{$\frac{1}{c}$};
	\draw[thick, dashed, domain=-1:0] plot(\x, 4);
	\draw[thick, dashed, domain=-1:2] plot(\x, 8);
	\draw (-1,0) node [below, font=\tiny]{$\frac{3}{4N}$};
	\draw (0,0) node [below, font=\tiny]{$\frac{1}{N}$};
	\draw (-1.25,0) node [left, font=\tiny]{$0$};
	\draw (-1,2) node [left, font=\tiny]{$\frac{1}{2N}$};
	\draw (-1,3) node [left, font=\tiny]{$\frac{3}{4N}$};
	\draw (-1,4) node [left, font=\tiny]{$\frac{1}{N}$};
	\draw (-1,8) node [left, font=\tiny]{$\frac{2}{N}$};
	\draw (0.5,0) node [below, font=\tiny]{$\frac{9}{8N}$};
	\draw (1,0) node [below, font=\tiny]{$\frac{5}{4N}$};
	\draw (4,0) node [below, font=\tiny]{$\frac{2}{N}$};
	\draw[thick, dashed, domain=0:8.25] plot (4,\x);
	\draw[thick, dashed, domain=0:8.25] plot (1,\x);
	\draw[thick, dashed, domain=0:8.25] plot (0.5,\x);
	\draw[thick, dashed, domain=-1:4.5] plot (\x,2);
	\draw[thick, dashed, domain=-1:4.5] plot (\x,3);
	\draw[thick, domain=-1.25:2.2] plot({\x}, {2*\x+4})
	node[above, font=\tiny]{$\frac{1}{c}=\frac{2}{a}-\frac{1}{N}$};
	\draw[thick, domain=-0.25:4.2] plot({\x}, {2*\x})
	node[above, font=\tiny]{$\frac{1}{c}=\frac{2}{a}-\frac{2}{N}$};
	\draw[ultra thick, color=white, domain=0:4] plot (0,\x);
	\draw[ultra thick, dashed, domain=0:4] plot (0,\x);
	\draw[ultra thick, color=white, domain=0:4] plot ({\x},{2*\x});
	\draw[ultra thick, dashed, domain=0:4] plot ({\x},{2*\x});
	\draw[ultra thick, domain=2:4] plot (\x,8);
	\draw[ultra thick, domain=0:2] plot ({\x},{2*\x+4});
	\draw[line width=1pt, fill=white] (0,0) circle (2pt);
	\draw[line width=1pt, fill=white] (4,8) circle (2pt);
	\draw[line width=1pt, fill=black] (2,8) circle (2pt);
	\draw[line width=1pt, fill=white] (0,4) circle (2pt);
	\draw[thick] (0.75,2.5) -- (1.25,1.5) node[draw, fill=white, inner sep=2pt, below, font=\tiny] {Range of $(1/q_1, 1/q_2)$};
\end{tikzpicture}
\caption{\footnotesize Range of $(1/a,1/c)$ where \eqref{eq:condition} and $a=b$ hold}
\end{subfigure}
\hfill
\begin{subfigure}{0.45\linewidth}
\centering
\begin{tikzpicture}
	\fill[fill=cyan!30] (0,4) -- (4,4) -- (0,8) -- cycle;
	\fill[fill=cyan!80] (2,4.5) -- (2,5) -- (3,5) -- (3,4.5) -- cycle;
	\draw[->,>=stealth,thick] (-0.5,0)--(5,0) 
	node [above, font=\small]{$\frac{1}{a}$};
	\draw[->,>=stealth,thick] (0,-0.5)--(0,8.5) 
	node [right, font=\small]{$\frac{1}{b}$};
	\draw (0,0) node [below left, font=\small]{$O$};
	\draw (0,4) node [left, font=\tiny]{$\frac{1}{N}$};
	\draw (0,8) node [left, font=\tiny]{$\frac{2}{N}$};
	\draw (4,0) node [below, font=\tiny]{$\frac{1}{N}$};
	\draw (2,0) node [below, font=\tiny]{$\frac{1}{2N}$};
	\draw (3,0) node [below, font=\tiny]{$\frac{3}{4N}$};
	\draw (0, 4.5) node [left, font=\tiny]{$\frac{9}{8N}$};
	\draw (0, 5) node [left, font=\tiny]{$\frac{5}{4N}$};
	\draw[thick, dashed, domain=0:4] plot (4,\x);
	\draw[thick, domain=-0.2:4.2] plot({\x}, {-\x+8})
	node[right, font=\tiny]{$\frac{1}{b}=-\frac{1}{a}+\frac{2}{N}$};
	\draw[thick, dashed, domain=0:4.5] plot (\x,5);
	\draw[thick, dashed, domain=0:4.5] plot (\x,4.5);
	\draw[thick, dashed, domain=0:8.25] plot (2,\x);
	\draw[thick, dashed, domain=0:8.25] plot (3,\x);
	\draw[ultra thick, color=white, domain=0:4] plot (\x,4);
	\draw[ultra thick, dashed, domain=0:4] plot (\x,4);
	\draw[ultra thick, color=white, domain=0:4] plot ({\x},{-\x+8});
	\draw[ultra thick, dashed, domain=0:4] plot ({\x},{-\x+8});
	\draw[ultra thick, domain=4:8] plot (0,\x);
	\draw[line width=1pt, fill=white] (0,4) circle (2pt);
	\draw[line width=1pt, fill=white] (0,8) circle (2pt);
	\draw[line width=1pt, fill=white] (4,4) circle (2pt);
	\draw[thick] (2.5,4.75) -- (3,3.75) node[draw, fill=white, inner sep=2pt, below, font=\tiny] {Range of $(1/q_2, 1/q_1)$};
\end{tikzpicture}
\caption{\footnotesize Range of $(1/a,1/b)$ where \eqref{eq:condition} and $c=\infty$ hold}
\end{subfigure}
\caption{Typical range of assumptions in Proposition~\ref{prop:bilinear} used in the proof of Theorem~\ref{thm:main}}
\label{fig:bilinear}
\end{figure}
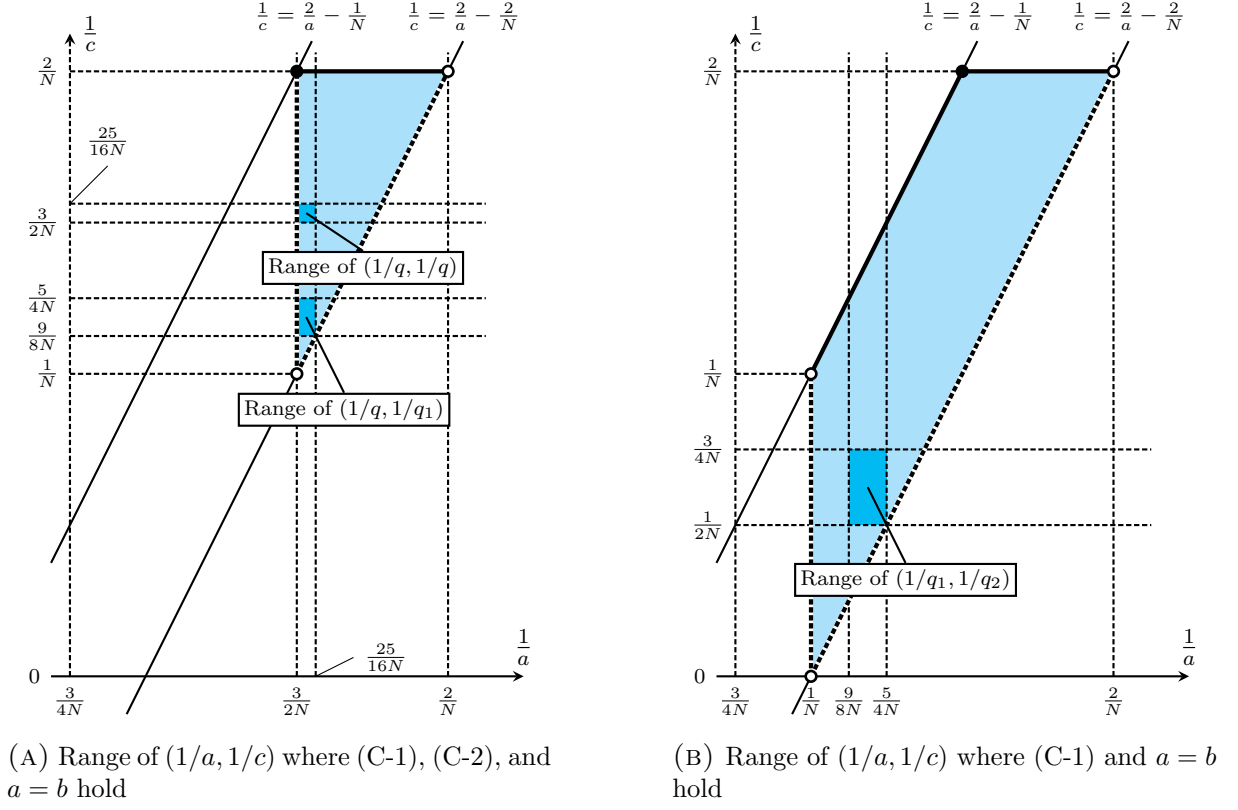
%------------------------------------

\begin{proof}
	We first recall several estimates which follow from \eqref{eq:condition}:
	\begin{alignat}{1}
	\label{eq:bilinear_est1-1}
		&0\le \dfrac{1}{c}\le \dfrac{1}{a}+\dfrac{1}{b}-\dfrac{1}{N}\le 1,\\
	\label{eq:bilinear_est1-3}
		&0\le \dfrac{1}{a}\le 1,\\
	\label{eq:bilinear_est1-2}
		&0<\dfrac{1}{b}-\dfrac{1}{N}<\dfrac{1}{b}<1,\\
	\label{eq:bilinear_est2}
		&\dfrac{2}{N}+\dfrac{1}{c}-\dfrac{1}{a}-\dfrac{1}{b}>0.	
	\end{alignat}
	The proof is divided into five steps.
	%--------------------------
	\\[5pt]
	\underline{\bf Step 1.}
	%--------------------------
	Fix $j \in \{ 1, 2 \}$ and $\alpha \in \mathbb{Z}^N_+$ satisfying $|\alpha|\le 1$ and $(j, |\alpha|)\neq(1,1)$ arbitrarily.
	By \eqref{eq:bilinear_est1-2}, Proposition~\ref{prop:HLS}, and the H\"{o}lder inequality we have
	\begin{alignat*}{1}
		\left\|
			f(s)\nabla K\ast g(s)
		\right\|_{(\frac{1}{a}+\frac{1}{b}-\frac{1}{N})^{-1}}
		\le \, &
		\|f(s)\|_{a}
		\|\nabla K\ast g(s)\|_{(\frac{1}{b}-\frac{1}{N})^{-1}}
		\\
		\le \, &
		C
		\|f(s)\|_{a}\|g(s)\|_{b}\\
		\le \, &
		C
		s^{-\frac{N}{2}(\frac{4}{N}-\frac{1}{a}-\frac{1}{b})}
		\norm{f}_{a, T}\norm{g}_{b, T}
	\end{alignat*}
	for $s\in(0, \min\{2, T\})$.
	Combining this with \eqref{lqlp_est;Fj} and \eqref{eq:bilinear_est1-1}, 
	we obtain
	\begin{equation}
	\label{eq:bilinear_est5}
	\begin{alignedat}{1}
		&\left\|
				\partial^\alpha_x \nabla\cdot
				[
					F_j(t-s)\ast 
					\left(
						f(s)\nabla K\ast g(s)
					\right)
				]
			\right\|_{c}\\
		\le \, &
		C
		(t-s)^{-\frac{N}{2j}
		\left(
			(
				\frac{1}{a}+\frac{1}{b}-\frac{1}{N}
			)
			-
			\frac{1}{c}
		\right)
		-\frac{|\alpha|+1}{2j}}
		s^{-\frac{N}{2}(\frac{4}{N}-\frac{1}{a}-\frac{1}{b})}
		\norm{f}_{a, T}\norm{g}_{b, T}\\
		= \, &
		C
		(t-s)^{\frac{N}{2j}(\frac{2j-|\alpha|}{N}+\frac{1}{c}-\frac{1}{a}-\frac{1}{b})-1}
		s^{\frac{N}{2}(\frac{1}{a}+\frac{1}{b}-\frac{2}{N})-1}
		\norm{f}_{a, T}\norm{g}_{b, T}
	\end{alignedat}
	\end{equation}
	for $t\in(0, T)$ and $s\in(0, \min\{2, t\}]$.
	Then it follows from \eqref{eq:bilinear_est2} that
	\begin{alignat*}{1}
		&
		\left\|
			\partial^\alpha_x
			\int^t_{\frac{t}{2}}
				\nabla \cdot
				[
					F_j(t-s)\ast
					\left(
						f(s)\nabla K\ast g(s)
					\right)
				]
			\,\mathrm{d}s
		\right\|_{c}\\
		\le \, &
		C
		\norm{f}_{a, T}\norm{g}_{b, T}
		\int^t_{\frac{t}{2}}
			(t-s)^{\frac{N}{2j}(\frac{2j-|\alpha|}{N}+\frac{1}{c}-\frac{1}{a}-\frac{1}{b})-1}
			s^{\frac{N}{2}(\frac{1}{a}+\frac{1}{b}-\frac{2}{N})-1}
		\,\mathrm{d}s\\
		\le \, &
		C
		\norm{f}_{a, T}\norm{g}_{b, T}
		t^{\frac{N}{2j}(\frac{2j-|\alpha|}{N}+\frac{1}{c}-\frac{1}{a}-\frac{1}{b})+\frac{N}{2}(\frac{1}{a}+\frac{1}{b}-\frac{2}{N})-1}
	\end{alignat*}
	for $t\in(0, \min\{2, T\})$.
	Since
	\begin{equation}
	\label{eq:exponent-1}
	\begin{alignedat}{1}
		&
		\dfrac{N}{2}
		\left(
			\dfrac{2}{N}-\dfrac{1}{c}
		\right)
		+\dfrac{|\alpha|}{2}
		+\dfrac{N}{2j}
		\left(
			\dfrac{2j-|\alpha|}{N}+\dfrac{1}{c}-\dfrac{1}{a}-\dfrac{1}{b}
		\right)
		+\dfrac{N}{2}
		\left(
			\dfrac{1}{a}+\dfrac{1}{b}-\dfrac{2}{N}
		\right)
		-1\\
		= \, &
		\dfrac{N(j-1)}{2j}
		\left(
			\dfrac{1}{a}+\dfrac{1}{b}+\dfrac{|\alpha|}{N}-\dfrac{1}{c}
		\right),
	\end{alignedat}
	\end{equation}
	we deduce from \eqref{eq:bilinear_est1-1} that
	\begin{equation}
	\label{eq:bilinear_est8}
		t^{\frac{N}{2}(\frac{2}{N}-\frac{1}{c})+\frac{|\alpha|}{2}}
		\left\|
			\partial^\alpha_x
			\int^t_{\frac{t}{2}}
				\nabla \cdot
				[
					F_j(t-s)\ast
					\left(
						f(s)\nabla K\ast g(s)
					\right)
				]\,\mathrm{d}s
		\right\|_{c}
		\le C\norm{f}_{a, T}\norm{g}_{b, T}
	\end{equation}
	for $t\in(0, \min\{2, T\})$, $j\in\{1,2\}$, and $\alpha\in\mathbb{Z}^N_{+}$ satisfying $|\alpha|\le1$ and $(j, |\alpha|)\neq(1,1)$.
	
	We next consider the case where \eqref{eq:condition} and \eqref{eq:additional} hold.
	By \eqref{eq:additional}, \eqref{eq:bilinear_est1-1}, \eqref{eq:bilinear_est2} and \eqref{eq:bilinear_est5} 
	we have
	\begin{alignat*}{1}
		&
		\left\|
			\partial^\alpha_x
			\int^{t}_{0}
				\nabla \cdot
				[
					F_j(t-s)\ast
					\left(
						f(s)\nabla K\ast g(s)
					\right)
				]\,\mathrm{d}s
		\right\|_{c}\\
		\le \, &
		C
		\norm{f}_{a, T}\norm{g}_{b, T}
			\int^{t}_{0}
				(t-s)^{\frac{N}{2j}(\frac{2j-|\alpha|}{N}+\frac{1}{c}-\frac{1}{a}-\frac{1}{b})-1}
				s^{\frac{N}{2}(\frac{1}{a}+\frac{1}{b}-\frac{2}{N})-1}
			\,\mathrm{d}s\\
		\le \, &
		C
		\norm{f}_{a, T}\norm{g}_{b, T}
		t^{\frac{N}{2j}(\frac{2j-|\alpha|}{N}+\frac{1}{c}-\frac{1}{a}-\frac{1}{b})+\frac{N}{2}(\frac{1}{a}+\frac{1}{b}-\frac{2}{N})-1}
	\end{alignat*}
	for $t\in(0, \min\{2, T\})$.
	It follows from this and \eqref{eq:exponent-1} that
	\begin{equation}
	\label{eq:bilinear_est9}
		t^{\frac{N}{2}(\frac{2}{N}-\frac{1}{c})+\frac{|\alpha|}{2}}
		\left\|
			\partial^\alpha_x
			\int^t_{0}
				\nabla \cdot
				[
					F_j(t-s)\ast
					\left(
						f(s)\nabla K\ast g(s)
					\right)
				]
			\,\mathrm{d}s
		\right\|_{c}
		\le 
		C
		\norm{f}_{a, T}\norm{g}_{b, T}
	\end{equation}
	for $t\in(0, \min\{2, T\})$, $j\in\{1,2\}$, and $\alpha\in\mathbb{Z}^N_{+}$ satisfying $|\alpha|\le1$ and $(j, |\alpha|)\neq(1,1)$. 
	Similarly, since
	\begin{alignat*}{1}
		&
		\left\|
			\partial^\alpha_x
			\int^{1}_{0}
				\nabla \cdot
				[
					F_j(t-s)\ast
					\left(
						f(s)\nabla K\ast g(s)
					\right)
				]
			\,\mathrm{d}s
		\right\|_{c}\\
		\le \, &
		C
		\norm{f}_{a, T}\norm{g}_{b, T}
		\int^{1}_{0}
			(t-s)^{\frac{N}{2j}(\frac{2j-|\alpha|}{N}+\frac{1}{c}-\frac{1}{a}-\frac{1}{b})-1}
			s^{\frac{N}{2}(\frac{1}{a}+\frac{1}{b}-\frac{2}{N})-1}
		\,\mathrm{d}s\\
		\le \, &
		C
		\norm{f}_{a, T}\norm{g}_{b, T}
		t^{\frac{N}{2j}(\frac{2j-|\alpha|}{N}+\frac{1}{c}-\frac{1}{a}-\frac{1}{b})-1}
		\int^{1}_{0}
			s^{\frac{N}{2}(\frac{1}{a}+\frac{1}{b}-\frac{2}{N})-1}
		\,\mathrm{d}s\\
		\le \, &
		C
		\norm{f}_{a, T}\norm{g}_{b, T}
		t^{\frac{N}{2j}(\frac{2j-|\alpha|}{N}+\frac{1}{c}-\frac{1}{a}-\frac{1}{b})-1}
	\end{alignat*}
	for $t\in(\min\{2, T\}, T)$ and
	\begin{alignat*}{1}
		&
		\dfrac{N}{4}
		\left(
			\dfrac{2}{N}-\dfrac{1}{c}
		\right)
		+\dfrac{|\alpha|}{4}
		+\dfrac{N}{2j}
		\left(
			\dfrac{2j-|\alpha|}{N}+\dfrac{1}{c}-\dfrac{1}{a}-\dfrac{1}{b}
		\right)
		-1\\
		= \, &
		-\dfrac{N}{4j}
		\left(
			\dfrac{2}{a}+\dfrac{2}{b}
			-\dfrac{2j}{N}-\dfrac{2-j}{c}
		\right),
	\end{alignat*}
	we deduce from \eqref{eq:additional} that
	\begin{equation}
	\label{eq:bilinear_est10}
		t^{\frac{N}{4}(\frac{2}{N}-\frac{1}{c})+\frac{|\alpha|}{4}}
		\left\|
			\partial^\alpha_x
			\int^{1}_{0}
				\nabla \cdot
				[
					F_j(t-s)\ast
					\left(
						f(s)\nabla K\ast g(s)
					\right)
				]
			\,\mathrm{d}s
		\right\|_{c}
		\le 
		C
		\norm{f}_{a, T}\norm{g}_{b, T},
	\end{equation}
	for $t\in(\min\{2, T\}, T)$, $j\in\{1,2\}$, and $\alpha\in\mathbb{Z}^N_{+}$ satisfying $|\alpha|\le1$ and $(j, |\alpha|)\neq(1,1)$. 
	%--------------------------
	\\[5pt]
	\underline{\bf Step 2.}
	%--------------------------
	Fix arbitrary $j \in \{ 1, 2 \}$ and $\alpha \in \mathbb{Z}^N_+$ satisfying $|\alpha|\le 1$ and $(j, |\alpha|) \neq (1,1)$.
	By \eqref{eq:bilinear_est1-3}, Propositions~\ref{prop:HLS}, \ref{prop:Sobolev}, and the H\"{o}lder inequality we have
	\begin{alignat*}{1}
		\left\|
			f(s)\nabla K\ast g(s)
		\right\|_{(\frac{1}{a}+\frac{1}{b}-\frac{1}{N})^{-1}}
		\le \, &
		\|f(s)\|_{(\frac{1}{b}-\frac{1}{N})^{-1}}
		\|K\ast \nabla g(s)\|_{a}\\
		\le \, &
		C
		\|\nabla f(s)\|_{b}
		\|\nabla g(s)\|_{a}\\
		\le \, &
		C
		s^{-\frac{N}{4}(\frac{4}{N}-\frac{1}{a}-\frac{1}{b})-\frac{1}{2}}
		\norm{f}_{b, T}\norm{g}_{a, T}
	\end{alignat*}
	for $s\in(\min\{1,T\}, T)$.
	This, together with \eqref{lqlp_est;Fj} and \eqref{eq:bilinear_est1-1},
	implies that
	\begin{equation}
	\label{eq:bilinear_est11}
	\begin{alignedat}{1}
		&
		\left\|
			\partial^\alpha_x\nabla\cdot
			[
				F_j(t-s)\ast 
				\left(
					f(s)\nabla K\ast g(s)
				\right)
			]
		\right\|_{c}\\
		\le \, &
		C
		(t-s)^{-\frac{N}{2j}
		\left(
			(\frac{1}{a}+\frac{1}{b}-\frac{1}{N})-\frac{1}{c}
		\right)-\frac{|\alpha|+1}{2j}}
		s^{-\frac{N}{4}(\frac{4}{N}-\frac{1}{a}-\frac{1}{b})-\frac{1}{2}}
		\norm{f}_{b, T}\norm{g}_{a, T}\\
		= \, &
		C
		(t-s)^{\frac{N}{2j}(\frac{2j-|\alpha|}{N}+\frac{1}{c}-\frac{1}{a}-\frac{1}{b})-1}
		s^{\frac{N}{4}(\frac{1}{a}+\frac{1}{b}-\frac{2}{N})-1}
		\norm{f}_{b, T}\norm{g}_{a, T}
	\end{alignedat}
	\end{equation}
	for $t\in(0, T)$ and $s\in(\min\{1, t\}, t]$.
	Then it follows from \eqref{eq:bilinear_est2} that
	\begin{alignat*}{1}
		&
		\left\|
			\partial^\alpha_x
			\int^t_{\frac{t}{2}}
				\nabla \cdot
				[
					F_j(t-s)\ast
					\left(
						f(s)\nabla K\ast g(s)
					\right)
				]
			\,\mathrm{d}s
		\right\|_{c}\\
		\le \, &
		C
		\norm{f}_{b, T}\norm{g}_{a, T}
		\int^t_{\frac{t}{2}}
			(t-s)^{\frac{N}{2j}(\frac{2j-|\alpha|}{N}+\frac{1}{c}-\frac{1}{a}-\frac{1}{b})-1}
			s^{\frac{N}{4}(\frac{1}{a}+\frac{1}{b}-\frac{2}{N})-1}
		\,\mathrm{d}s\\
		\le \, &
		C
		\norm{f}_{b, T}\norm{g}_{a, T}
		t^{\frac{N}{2j}(\frac{2j-|\alpha|}{N}+\frac{1}{c}-\frac{1}{a}-\frac{1}{b})+\frac{N}{4}(\frac{1}{a}+\frac{1}{b}-\frac{2}{N})-1}
	\end{alignat*}
	for $t\in(\min\{2, T\}, T)$.
	Since
	\begin{equation}
	\label{eq:exponent-2}
	\begin{alignedat}{1}
		&
		\dfrac{N}{4}
		\left(
			\dfrac{2}{N}-\dfrac{1}{c}
		\right)
		+\dfrac{|\alpha|}{4}
		+\dfrac{N}{2j}
		\left(
			\dfrac{2j-|\alpha|}{N}+\dfrac{1}{c}-\dfrac{1}{a}-\dfrac{1}{b}
		\right)
		+\dfrac{N}{4}
		\left(
			\dfrac{1}{a}+\dfrac{1}{b}-\dfrac{2}{N}
		\right)
		-1\\
		= \, &
		-\dfrac{N(2-j)}{4j}
		\left(
			\dfrac{1}{a}+\dfrac{1}{b}-\dfrac{1}{c}
		\right),
	\end{alignedat}	
	\end{equation}
	we deduce from \eqref{eq:bilinear_est1-1} that
	\begin{equation}
	\label{eq:bilinear_est13}
	\begin{alignedat}{1}
		t^{\frac{N}{4}(\frac{2}{N}-\frac{1}{c})+\frac{|\alpha|}{4}}
		\left\|
			\partial^\alpha_x
			\int^t_{\frac{t}{2}}
				\nabla \cdot
				[
					F_j(t-s)\ast
					\left(
						f(s)\nabla K\ast g(s)
					\right)
				]
			\, \mathrm{d}s
		\right\|_{c}
		\le C\norm{f}_{b, T}\norm{g}_{a, T}
	\end{alignedat}
	\end{equation}
	for $t\in(\min\{2, T\}, T)$, $j\in\{1,2\}$, and $\alpha\in\mathbb{Z}^N_{+}$ satisfying $|\alpha|\le1$ and $(j, |\alpha|)\neq(1,1)$.
	
	We next consider the case \eqref{eq:condition} and \eqref{eq:additional} hold. 
	By \eqref{eq:additional}, \eqref{eq:bilinear_est1-1}, \eqref{eq:bilinear_est2} and \eqref{eq:bilinear_est11} 
	we have
	\begin{alignat*}{1}
		&
		\left\|
			\partial^\alpha_{x}
			\int^{t}_{1}
				\nabla \cdot
				[
					F_j(t-s)\ast
					\left(
						f(s)\nabla K\ast g(s)
					\right)
				]
			\,\mathrm{d}s
		\right\|_{c}\\
		\le \, &
		C
		\norm{f}_{b, T}\norm{g}_{a, T}
		\int^{t}_{1}
			(t-s)^{\frac{N}{2j}(\frac{2j-|\alpha|}{N}+\frac{1}{c}-\frac{1}{a}-\frac{1}{b})-1}
			s^{\frac{N}{4}(\frac{1}{a}+\frac{1}{b}-\frac{2}{N})-1}
		\,\mathrm{d}s\\
		\le \, &
		C
		\norm{f}_{b, T}\norm{g}_{a, T}
		t^{\frac{N}{2j}(\frac{2j-|\alpha|}{N}+\frac{1}{c}-\frac{1}{a}-\frac{1}{b})+\frac{N}{4}(\frac{1}{a}+\frac{1}{b}-\frac{2}{N})-1}
	\end{alignat*}
	for $t\in(\min\{2, T\}, T)$.
	This, together with \eqref{eq:exponent-2}, implies that
	\begin{equation}
	\label{eq:bilinear_est14}
		t^{\frac{N}{4}(\frac{2}{N}-\frac{1}{c})+\frac{|\alpha|}{4}}
		\left\|
			\partial_x^\alpha
			\int^t_{1}
				\nabla \cdot
				[
					F_j(t-s)\ast
					\left(
						f(s)\nabla K\ast g(s)
					\right)
				]
			\,\mathrm{d}s
		\right\|_{c}
		\le 
		C
		\norm{f}_{b, T}\norm{g}_{a, T}
	\end{equation}
	for $t\in(\min\{2, T\}, T)$, $j\in\{1,2\}$, and $\alpha\in\mathbb{Z}^N_{+}$ satisfying $|\alpha|\le1$ and $(j, |\alpha|)\neq(1,1)$. 
	%--------------------------
	\\[5pt]
	\underline{\bf Step 3.}
	%--------------------------
	Fix $\alpha\in\mathbb{Z}^{N}_+$ with $|\alpha|=1$ arbitrarily.
	By \eqref{eq:bilinear_est1-2}, Propositions~\ref{prop:HLS}, \ref{prop:CZ}, and \ref{prop:Sobolev}, and the H\"{o}lder inequality, we have
	\begin{alignat*}{1}
		&
		\left\|
			\nabla\cdot
			\left(
				f(s)\nabla K\ast g(s)
			\right)
		\right\|_{(\frac{1}{a}+\frac{1}{b}-\frac{1}{N})^{-1}}\\
		= \, &
		\left\|
			f(s)\Delta K\ast g(s)
			+\nabla f(s) \cdot \nabla K\ast g(s)
		\right\|_{(\frac{1}{a}+\frac{1}{b}-\frac{1}{N})^{-1}}\\
		\le \, &
		\|f(s)\|_{(\frac{1}{b}-\frac{1}{N})^{-1}}
		\|\Delta K\ast g(s)\|_{a}+\|\nabla f(s)\|_{a}\|\nabla K\ast g(s)\|_{(\frac{1}{b}-\frac{1}{N})^{-1}}
		\\
		\le \, &
		C
		\left(
			\|\nabla f(s)\|_{b}\|g(s)\|_{a}
			+\|\nabla f(s)\|_{a}\|g(s)\|_{b}
		\right)\\
		\le \, &
		C
		s^{-\frac{N}{2}(\frac{4}{N}-\frac{1}{a}-\frac{1}{b})-\frac{1}{2}}
		\left(
			\norm{f}_{b, T}\norm{g}_{a, T}+\norm{f}_{a, T}\norm{g}_{b, T}
		\right)
	\end{alignat*}
	for $s\in (0, \min\{2, T\})$.
	This, together with \eqref{lqlp_est;Fj} and \eqref{eq:bilinear_est1-1}, 
	implies that
	\begin{equation}
	\label{eq:bilinear_est17}	
	\begin{alignedat}{1}
		&
		\left\|
			\partial^\alpha_x\nabla \cdot
			[
				F_1(t-s)\ast 
				\left(
					f(s)\nabla K\ast g(s)
				\right)
			]
		\right\|_{c}\\
		= \, &
		\left\|
			\partial^\alpha_x
			[
				F_1(t-s)\ast \nabla\cdot
				\left(
					f(s)\nabla K\ast g(s)
				\right)
			]
		\right\|_{c}\\
		\le \, &
		C
		(t-s)^{-\frac{N}{2}
		\left(
			(\frac{1}{a}+\frac{1}{b}-\frac{1}{N})-\frac{1}{c}
		\right)
		-\frac{1}{2}}
		s^{-\frac{N}{2}(\frac{4}{N}-\frac{1}{a}-\frac{1}{b})-\frac{1}{2}}
		\left(
			\norm{f}_{b, T}\norm{g}_{a, T}
			+\norm{f}_{a, T}\norm{g}_{b, T}
		\right)\\
		=\, &
		C
		(t-s)^{\frac{N}{2}(\frac{2}{N}+\frac{1}{c}-\frac{1}{a}-\frac{1}{b})-1}
		s^{\frac{N}{2}(\frac{1}{a}+\frac{1}{b}-\frac{3}{N})-1}
		\left(
			\norm{f}_{b, T}\norm{g}_{a, T}+\norm{f}_{a, T}\norm{g}_{b, T}
		\right)
	\end{alignedat}
	\end{equation}
	for $t\in(0, T)$ and $s\in(0, \min\{2, t\})$.
	Then it follows from \eqref{eq:bilinear_est2} that
	\begin{alignat*}{1}
		&
		\left\|
			\partial^\alpha_x
			\int^t_{\frac{t}{2}}
				\nabla \cdot
				[
					F_1(t-s)\ast
					\left(
						f(s)\nabla K\ast g(s)
					\right)
				]
			\,\mathrm{d}s
		\right\|_{c}\\
		\le \, &
		C
		\left(
			\norm{f}_{b, T}\norm{g}_{a, T}
			+\norm{f}_{a, T}\norm{g}_{b, T}
		\right)
		\int^t_{\frac{t}{2}}
			(t-s)^{\frac{N}{2}(\frac{2}{N}+\frac{1}{c}-\frac{1}{a}-\frac{1}{b})-1}
			s^{\frac{N}{2}(\frac{1}{a}+\frac{1}{b}-\frac{3}{N})-1}
		\,\mathrm{d}s\\
		\le \, &
		C
		\left(
			\norm{f}_{b, T}\norm{g}_{a, T}
			+\norm{f}_{a, T}\norm{g}_{b, T}
		\right)
		t^{\frac{N}{2}(\frac{2}{N}+\frac{1}{c}-\frac{1}{a}-\frac{1}{b})+\frac{N}{2}(\frac{1}{a}+\frac{1}{b}-\frac{3}{N})-1}
	\end{alignat*}
	for $t\in(0, \min\{2, T\})$.
	Since
	\begin{equation}
	\label{eq:exponent-3}
		\dfrac{N}{2}
		\left(
			\dfrac{2}{N}-\dfrac{1}{c}
		\right)
		+\dfrac{1}{2}
		+\dfrac{N}{2}
		\left(
			\dfrac{2}{N}+\dfrac{1}{c}-\dfrac{1}{a}-\dfrac{1}{b}
		\right)
		+\dfrac{N}{2}
		\left(
			\dfrac{1}{a}+\dfrac{1}{b}-\dfrac{3}{N}
		\right)
		-1=0,
	\end{equation}
	we obtain
	\begin{equation}
	\label{eq:bilinear_est18}
	\begin{alignedat}{1}
		&
		t^{\frac{N}{2}(\frac{2}{N}-\frac{1}{c})+\frac{1}{2}}
		\left\|
			\partial^\alpha_x
			\int^t_{\frac{t}{2}}
				\nabla \cdot
				[
					F_1(t-s)\ast
					\left(
						f(s)\nabla K\ast g(s)
					\right)
				]
			\,\mathrm{d}s
		\right\|_{c}\\
		\le \, &
		C
		\left(
			\norm{f}_{b, T}\norm{g}_{a, T}
			+\norm{f}_{a, T}\norm{g}_{b, T}
		\right)
	\end{alignedat}	
	\end{equation}
	for $t\in(0, \min\{2, T\})$ and $\alpha\in\mathbb{Z}^N_+$ satisfying $|\alpha|=1$.
	
	We next consider the case \eqref{eq:condition} and \eqref{eq:additional} hold.
	By \eqref{eq:additional}, \eqref{eq:bilinear_est1-1}, \eqref{eq:bilinear_est2} and \eqref{eq:bilinear_est17}
	we have
	\begin{alignat*}{1}
		&
		\left\|
			\partial^\alpha_x
			\int^t_0
				\nabla \cdot
				[
					F_1(t-s)\ast
					\left(
						f(s)\nabla K\ast g(s)
					\right)
				]
			\,\mathrm{d}s
		\right\|_{c}\\
		\le \, &
		C
		\left(
			\norm{f}_{b, T}\norm{g}_{a, T}
			+\norm{f}_{a, T}\norm{g}_{b, T}
		\right)
		\int^t_0
			(t-s)^{\frac{N}{2}(\frac{2}{N}+\frac{1}{c}-\frac{1}{a}-\frac{1}{b})-1}
			s^{\frac{N}{2}(\frac{1}{a}+\frac{1}{b}-\frac{3}{N})-1}
		\,\mathrm{d}s\\
		\le \, &
		C
		\left(
			\norm{f}_{b, T}\norm{g}_{a, T}
			+\norm{f}_{a, T}\norm{g}_{b, T}
		\right)
		t^{\frac{N}{2}(\frac{2}{N}+\frac{1}{c}-\frac{1}{a}-\frac{1}{b})+\frac{N}{2}(\frac{1}{a}+\frac{1}{b}-\frac{3}{N})-1}
	\end{alignat*}
	for $t\in(0, \min\{2, T\})$.
	This, together with \eqref{eq:exponent-3}, implies that
	\begin{equation}
	\label{eq:bilinear_est19}
	\begin{alignedat}{1}
		&
		t^{\frac{N}{2}(\frac{2}{N}-\frac{1}{c})+\frac{1}{2}}
		\left\|
			\partial^\alpha_x
			\int^t_{0}
				\nabla \cdot
				[
					F_1(t-s)\ast
					\left(
						f(s)\nabla K\ast g(s)
					\right)
				]
			\,\mathrm{d}s
		\right\|_{c}\\
		\le \, & 
		C
		\left(
			\norm{f}_{b, T}\norm{g}_{a, T}
			+\norm{f}_{a, T}\norm{g}_{b, T}
		\right)
	\end{alignedat}
	\end{equation}
	for $t\in(0, \min\{2, T\})$ and $\alpha\in\mathbb{Z}^N_{+}$ satisfying $|\alpha|=1$,
	if \eqref{eq:additional} holds.
	Similarly, since
	\begin{alignat*}{1}
		&
		\left\|
			\partial^\alpha_x
			\int^{1}_{0}
				\nabla \cdot
				[
					F_1(t-s)\ast
					\left(
						f(s)\nabla K\ast g(s)
					\right)
				]
			\,\mathrm{d}s
		\right\|_{c}\\
		\le \, &
		C
		\left(
			\norm{f}_{b, T}\norm{g}_{a, T}
			+\norm{f}_{a, T}\norm{g}_{b, T}
		\right)
		\int^1_0
			(t-s)^{\frac{N}{2}(\frac{2}{N}+\frac{1}{c}-\frac{1}{a}-\frac{1}{b})-1}
			s^{\frac{N}{2}(\frac{1}{a}+\frac{1}{b}-\frac{3}{N})-1}
		\,\mathrm{d}s\\
		\le \, &
		C
		\left(
			\norm{f}_{b, T}\norm{g}_{a, T}
			+\norm{f}_{a, T}\norm{g}_{b, T}
		\right)
		t^{\frac{N}{2}(\frac{2}{N}+\frac{1}{c}-\frac{1}{a}-\frac{1}{b})-1}
		\int^{1}_{0}
			s^{\frac{N}{2}(\frac{1}{a}+\frac{1}{b}-\frac{3}{N})-1}
		\,\mathrm{d}s\\
		\le \, &
		C
		\left(
			\norm{f}_{b, T}\norm{g}_{a, T}
			+\norm{f}_{a, T}\norm{g}_{b, T}
		\right)
		t^{\frac{N}{2}(\frac{2}{N}+\frac{1}{c}-\frac{1}{a}-\frac{1}{b})-1}
	\end{alignat*}
	for $t\in(\min\{2, T\}, T)$ and
	\[
		\dfrac{N}{4}
		\left(
			\dfrac{2}{N}-\dfrac{1}{c}
		\right)
		+\dfrac{1}{4}
		+\dfrac{N}{2}
		\left(
			\dfrac{2}{N}+\dfrac{1}{c}-\dfrac{1}{a}-\dfrac{1}{b}
		\right)
		-1
		=
		-\dfrac{N}{4}
		\left(
			\dfrac{2}{a}+\dfrac{2}{b}
			-\dfrac{3}{N}-\dfrac{1}{c}
		\right),
	\]
	we deduce from \eqref{eq:additional} that
	\begin{equation}
	\label{eq:bilinear_est20}
		t^{\frac{N}{4}(\frac{2}{N}-\frac{1}{c})+\frac{1}{4}}
		\left\|
			\partial^\alpha_x
			\int^{1}_{0}
				\nabla \cdot
				[
					F_1(t-s)\ast
					\left(
						f(s)\nabla K\ast g(s)
					\right)
				]
			\,\mathrm{d}s
		\right\|_{c}
		\le 
		C
		\norm{f}_{a, T}\norm{g}_{b, T},
	\end{equation}
	for $t\in(\min\{2, T\}, T)$ and $\alpha\in\mathbb{Z}^N_{+}$ satisfying $|\alpha|=1$. 
	%--------------------------
	\\[5pt]
	\underline{\bf Step 4.}
	%--------------------------
	Fix $\alpha\in\mathbb{Z}^N_+$ with $|\alpha|= 1$ arbitrarily.
	By \eqref{eq:bilinear_est1-3}, \eqref{eq:bilinear_est1-2}, Propositions~\ref{prop:HLS}, \ref{prop:Sobolev}, and the H\"{o}lder inequality we have
	\begin{alignat*}{1}
		&
		\left\|
			\nabla\cdot\left(f(s)\nabla K\ast g(s)\right)
		\right\|_{(\frac{1}{a}+\frac{1}{b}-\frac{1}{N})^{-1}}\\
		= \, &
		\left\|
			f(s)\Delta K\ast g(s)
			+\nabla f(s) \cdot \nabla K\ast g(s)
		\right\|_{(\frac{1}{a}+\frac{1}{b}-\frac{1}{N})^{-1}}\\
		\le \, &
		\|f(s)\|_{(\frac{1}{b}-\frac{1}{N})^{-1}}\|\nabla K\ast \nabla g(s)\|_{a}
		+\|\nabla f(s)\|_{a}\|\nabla K\ast g(s)\|_{(\frac{1}{b}-\frac{1}{N})^{-1}}\\
		\le \, &
		C
		\left(
			\|\nabla f(s)\|_{b}\|\nabla g(s)\|_{a}
			+\|\nabla f(s)\|_{a}\|g(s)\|_{(\frac{1}{b}-\frac{1}{N})^{-1}}
		\right)\\
		\le \, &
		C
		\left(
			\|\nabla f(s)\|_{b}\|\nabla g(s)\|_{a}
			+\|\nabla f(s)\|_{a}\|\nabla g(s)\|_{b}
		\right)\\
		\le \, &
		C
		s^{-\frac{N}{4}(\frac{4}{N}-\frac{1}{a}-\frac{1}{b})-\frac{1}{2}}
		\left(
			\norm{f}_{b, T}\norm{g}_{a, T}+\norm{f}_{a, T}\norm{g}_{b, T}
		\right)
	\end{alignat*}
	for $s\in(\min\{1,T\}, T)$.
	This, together with \eqref{lqlp_est;Fj} and \eqref{eq:bilinear_est1-1}, 
	implies that
	\begin{equation}
	\label{eq:bilinear_est21}
	\begin{alignedat}{1}
		&
		\left\|
			\partial^\alpha_x\nabla\cdot
			[
				F_1(t-s)\ast 
				\left(
					f(s)\nabla K\ast g(s)
				\right)
			]
		\right\|_{c}\\
		\le \, &
		C
		(t-s)^{-\frac{N}{2}
		\left(
			(\frac{1}{a}+\frac{1}{b}-\frac{1}{N})-\frac{1}{c}
		\right)-\frac{1}{2}}
		s^{-\frac{N}{4}(\frac{4}{N}-\frac{1}{a}-\frac{1}{b})-\frac{1}{2}}
		\left(
			\norm{f}_{b, T}\norm{g}_{a, T}+\norm{f}_{a, T}\norm{g}_{b, T}
		\right)\\
		= \, &
		C
		(t-s)^{\frac{N}{2}(\frac{2}{N}+\frac{1}{c}-\frac{1}{a}-\frac{1}{b})-1}
		s^{\frac{N}{4}(\frac{1}{a}+\frac{1}{b}-\frac{2}{N})-1}
		\left(
			\norm{f}_{b, T}\norm{g}_{a, T}+\norm{f}_{a, T}\norm{g}_{b, T}
		\right)
	\end{alignedat}
	\end{equation}
	for $t\in(0, T)$ and $s\in(\min\{1, t\}, t]$.
	Then it follows from \eqref{eq:bilinear_est2} that
	\begin{alignat*}{1}
		&
		\left\|
			\partial^\alpha_x
			\int^t_{\frac{t}{2}}
				\nabla \cdot
				[
					F_1(t-s)\ast
					\left(
						f(s)\nabla K\ast g(s)
					\right)
				]
			\,\mathrm{d}s
		\right\|_{c}\\
		\le \, &
		C
		\left(
			\norm{f}_{b, T}\norm{g}_{a, T}+\norm{f}_{a, T}\norm{g}_{b, T}
		\right)
		\int^t_{\frac{t}{2}}
			(t-s)^{\frac{N}{2}(\frac{2}{N}+\frac{1}{c}-\frac{1}{a}-\frac{1}{b})-1}
			s^{\frac{N}{4}(\frac{1}{a}+\frac{1}{b}-\frac{2}{N})-1}
		\,\mathrm{d}s\\
		\le \, &
		C
		\left(
			\norm{f}_{b, T}\norm{g}_{a, T}+\norm{f}_{a, T}\norm{g}_{b, T}
		\right)
		t^{\frac{N}{2}(\frac{2}{N}+\frac{1}{c}-\frac{1}{a}-\frac{1}{b})+\frac{N}{4}(\frac{1}{a}+\frac{1}{b}-\frac{2}{N})-1}
	\end{alignat*}
	for $t\in(\min\{2, T\}, T)$.
	Since
	\begin{equation}
	\label{eq:exponent-4}
	\begin{alignedat}{1}
		&
		\dfrac{N}{4}
		\left(
			\dfrac{2}{N}-\dfrac{1}{c}
		\right)
		+\dfrac{1}{4}
		+\dfrac{N}{2}
		\left(
			\dfrac{2}{N}+\dfrac{1}{c}-\dfrac{1}{a}-\dfrac{1}{b}
		\right)
		+\dfrac{N}{4}
		\left(
			\dfrac{1}{a}+\dfrac{1}{b}-\dfrac{2}{N}
		\right)
		-1\\
		= \, &
		-\dfrac{N}{4}
		\left(
			\dfrac{1}{a}+\dfrac{1}{b}
			-\dfrac{1}{N}-\dfrac{1}{c}
		\right),
	\end{alignedat}	
	\end{equation}
	we deduce from \eqref{eq:bilinear_est1-1} that
	\begin{equation}
	\label{eq:bilinear_est22}
	\begin{alignedat}{1}
		&
		t^{\frac{N}{4}(\frac{2}{N}-\frac{1}{c})+\frac{1}{4}}
		\left\|
			\partial^\alpha_x
			\int^t_{\frac{t}{2}}
				\nabla \cdot
				[
					F_1(t-s)\ast
					\left(
						f(s)\nabla K\ast g(s)
					\right)
				]
			\, \mathrm{d}s
		\right\|_{c} \\
		&\le 
		C
		\left(
			\norm{f}_{b, T}\norm{g}_{a, T}+\norm{f}_{a, T}\norm{g}_{b, T}
		\right)
	\end{alignedat}
	\end{equation}
	for $t\in(\min\{2, T\}, T)$ and $\alpha\in\mathbb{Z}^N_{+}$ satisfying $|\alpha|=1$.
	
	We next consider the case \eqref{eq:condition} and \eqref{eq:additional} hold.
	By \eqref{eq:additional}, \eqref{eq:bilinear_est1-1}, \eqref{eq:bilinear_est2} and \eqref{eq:bilinear_est21} 
	we have
	\begin{alignat*}{1}
		&
		\left\|
			\partial^\alpha_{x}
			\int^{t}_{1}
				\nabla \cdot
				[
					F_1(t-s)\ast
					\left(
						f(s)\nabla K\ast g(s)
					\right)
				]
			\,\mathrm{d}s
		\right\|_{c}\\
		\le \, &
		C
		\left(
			\norm{f}_{b, T}\norm{g}_{a, T}+\norm{f}_{a, T}\norm{g}_{b, T}
		\right)
		\int^{t}_{1}
			(t-s)^{\frac{N}{2}(\frac{2}{N}+\frac{1}{c}-\frac{1}{a}-\frac{1}{b})-1}
			s^{\frac{N}{4}(\frac{1}{a}+\frac{1}{b}-\frac{2}{N})-1}
		\,\mathrm{d}s\\
		\le \, &
		C
		\left(
			\norm{f}_{b, T}\norm{g}_{a, T}+\norm{f}_{a, T}\norm{g}_{b, T}
		\right)
		t^{\frac{N}{2}(\frac{2}{N}+\frac{1}{c}-\frac{1}{a}-\frac{1}{b})+\frac{N}{4}(\frac{1}{a}+\frac{1}{b}-\frac{2}{N})-1}
	\end{alignat*}
	for $t\in(\min\{2, T\}, T)$.
	This together with \eqref{eq:exponent-4} implies that
	\begin{equation}
	\label{eq:bilinear_est24}
	\begin{alignedat}{1}
		&
		t^{\frac{N}{4}(\frac{2}{N}-\frac{1}{c})+\frac{1}{4}}
		\left\|
			\partial_x^\alpha
			\int^t_{1}
				\nabla \cdot
				[
					F_1(t-s)\ast
					\left(
						f(s)\nabla K\ast g(s)
					\right)
				]
			\,\mathrm{d}s
		\right\|_{c}\\
		\le \, &
		C
		\left(
			\norm{f}_{b, T}\norm{g}_{a, T}+\norm{f}_{a, T}\norm{g}_{b, T}
		\right)
	\end{alignedat}
	\end{equation}
	for $t\in(\min\{2, T\}, T)$ and $\alpha\in\mathbb{Z}^N_{+}$ satisfying $|\alpha|=1$. 
	%--------------------------
	\\[5pt]
	\underline{\bf Step 5.}
	%--------------------------
	We summarize the estimates obtained in the previous steps and complete the proof of Proposition~\ref{prop:bilinear}.
	If \eqref{eq:condition} holds, then by \eqref{eq:bilinear_est8}, \eqref{eq:bilinear_est13}, \eqref{eq:bilinear_est18}, and \eqref{eq:bilinear_est22} we have
	\begin{alignat*}{1}
		&
		\norm{\Theta_1[f,g]}_{c, T}\\
		\le \, &
		C
		\sum_{|\alpha|\le 1}
		\sum_{j\in\{1,2\}}
		\sup_{t\in(0, \min\{2, T\})}
		t^{\frac{N}{2}(\frac{2}{N}-\frac{1}{c})+\frac{|\alpha|}{2}}
		\left\|
			\partial^\alpha_x
			\int^t_{\frac{t}{2}}
				\nabla \cdot
				[
					F_j(t-s)\ast
					\left(
						f(s)\nabla K\ast g(s)
					\right)
				]
			\,\mathrm{d}s
		\right\|_{c}\\
		& \,
		+C
		\sum_{|\alpha|\le 1}
		\sum_{j\in\{1,2\}}
		\sup_{t\in(\min\{2, T\}, T)}
		t^{\frac{N}{4}(\frac{2}{N}-\frac{1}{c})+\frac{|\alpha|}{4}}
		\left\|
			\partial^\alpha_x
			\int^t_{\frac{t}{2}}
				\nabla \cdot
				[
					F_j(t-s)\ast
					\left(
						f(s)\nabla K\ast g(s)
					\right)
				]\,\mathrm{d}s
		\right\|_{c}\\
		\le \, &
		C
		\left(\norm{f}_{a, T}+\norm{f}_{b, T}\right)
		\left(\norm{g}_{a, T}+\norm{g}_{b, T}\right).
	\end{alignat*}
	Thus we obtain the desired estimate in the case where \eqref{eq:condition} holds.
	On the other hand, if \eqref{eq:condition} and \eqref{eq:additional} hold, then by \eqref{eq:bilinear_est9}, \eqref{eq:bilinear_est10}, \eqref{eq:bilinear_est14}, \eqref{eq:bilinear_est19}, \eqref{eq:bilinear_est20}, and \eqref{eq:bilinear_est24} we have
	\begin{alignat*}{1}
		&
		\norm{\Theta_2[f,g]}_{c, T}\\
		\le \, &
		C
		\sum_{|\alpha|\le 1}
		\sum_{j\in\{1,2\}}
		\sup_{t\in(0, \min\{2, T\})}
		t^{\frac{N}{2}(\frac{2}{N}-\frac{1}{c})+\frac{|\alpha|}{2}}
		\left\|
			\partial^\alpha_x
			\int^t_{0}
				\nabla \cdot
				[
					F_j(t-s)\ast
					\left(
						f(s)\nabla K\ast g(s)
					\right)
				]
			\,\mathrm{d}s
		\right\|_{c}\\
		& \,
		+C
		\sum_{|\alpha|\le 1}
		\sum_{j\in\{1,2\}}
		\sup_{t\in(\min\{2, T\}, T)}
		t^{\frac{N}{4}(\frac{2}{N}-\frac{1}{c})+\frac{|\alpha|}{4}}
		\left\|
			\partial^\alpha_x
			\int^1_{0}
				\nabla \cdot
				[
					F_j(t-s)\ast
					\left(
						f(s)\nabla K\ast g(s)
					\right)
				]\,\mathrm{d}s
		\right\|_{c}\\
		& \,
		+C
		\sum_{|\alpha|\le 1}
		\sum_{j\in\{1,2\}}
		\sup_{t\in(\min\{2, T\}, T)}
		t^{\frac{N}{4}(\frac{2}{N}-\frac{1}{c})+\frac{|\alpha|}{4}}
		\left\|
			\partial^\alpha_x
			\int^t_{1}
				\nabla \cdot
				[
					F_j(t-s)\ast
					\left(
						f(s)\nabla K\ast g(s)
					\right)
				]\,\mathrm{d}s
		\right\|_{c}\\
		\le \, &
		C
		\left(\norm{f}_{a, T}+\norm{f}_{b, T}\right)
		\left(\norm{g}_{a, T}+\norm{g}_{b, T}\right).
	\end{alignat*}
	Thus we obtain the desired estimate in the case where \eqref{eq:condition} and \eqref{eq:additional} hold.
	Moreover, the continuity of $\Theta_1[f,g]$ and $\Theta_2[f,g]$ follows from Corollary~\ref{cor:semigroup}-(2), estimates obtained in the previous steps, and the argument in Weissler~\cite[Lemma~2.1]{W1979}, and hence we omit the proof.
	Therefore, we complete the proof of Proposition~\ref{prop:bilinear}.
\end{proof}

%============================================================================%
\subsection{Proof of Theorem~\ref{thm:main}}
\label{subsec:proof}
%============================================================================%
We are now in a position to complete the proof of Theorem~\ref{thm:main}.
The proof is based on the argument in Brezis--Cazenave~\cite{BC1996}.
\begin{proof}[Proof of Theorem~\ref{thm:main}]
	Let $q$ satisfy
	\begin{equation}
	\label{eq:q}
		\dfrac{3}{2N}<\dfrac{1}{q}<\dfrac{25}{16N}.
	\end{equation}
	For $T>0$ and $f\in X_{q, T}$ we define
	\[
		\Psi[f](t)
		\coloneqq
		F(t)\ast v_0
		-\Theta_2[f,f](t)
		=
		F(t)\ast v_0
		-\int^t_0
		\nabla\cdot
		[
			F(t-s)\ast 
			\left(
				f(s)\nabla K\ast f(s)
			\right)
		]
		\,\mathrm{d}s
	\]
	for $t\in(0, T)$.
	Since the choice $a=b=c=q$ satisfies \eqref{eq:condition} and \eqref{eq:additional} in Proposition~\ref{prop:bilinear} (see Figure~\ref{fig:bilinear}-(a)),
	we have
	\begin{alignat}{1}
	\label{eq:contra-0}
		&
		\Psi[f]\in X_{q,T},\\
	\label{eq:contra-1}
		&
		\norm{\Psi[f]}_{q, T}
		\le 
		\norm{F(t)\ast v_0}_{q, T}
		+C
		\norm{f}_{q, T}^2,\\
	\label{eq:contra-2}
		&\norm{\Psi[f]-\Psi[g]}_{q, T}
		\le 
		C
		\left(\norm{f}_{q, T}+\norm{g}_{q, T}\right)
		\norm{f-g}_{q, T},
	\end{alignat}
	for $T>0$ and $f,g\in X_{q, T}$.
	In particular, combining \eqref{eq:contra-1} with Proposition~\ref{prop:semigroup} and Corollary~\ref{cor:semigroup}-(3), we obtain
	\begin{alignat}{1}
	\label{eq:contra-3}
		&\lim_{t\searrow0}
		\norm{\Psi[f]}_{q,t}
		=0
		\quad
		\text{if}
		\quad
		\lim_{t\searrow0}
		\norm{f}_{q,t}
		=0,\\
	\label{eq:contra-4}
		&
		\norm{\Psi[f]}_{q, T}
		\le 
		C
		\|v_0\|_{\frac{N}{2}}
		+C
		\norm{f}_{q, T}^2,
	\end{alignat}
	for $T>0$ and $f\in X_{q, T}$, respectively.
	
	We construct a function $v$ satisfying \eqref{eq:IE}
	as a fixed point of $\Psi$.
	Fix $T>0$ arbitrarily.
	Let $\varepsilon_*>0$ and $c_*>0$ be constants chosen later and set
	\[
		Y_{T}
		\coloneqq
		\left\{f
			\in X_{q, T}
			\,\middle|\, 
			\norm{f}_{q, T}
			\le c_*\|v_0\|_{\frac{N}{2}},
			\lim_{t\searrow 0}\norm{f}_{q, t}=0
		\right\},
		\quad
		d_{T}(f, g)
		\coloneqq
		\norm{f-g}_{q, T}.
	\]
	Then $(Y_T, d_T)$ is a complete metric space.
	Moreover, by \eqref{eq:initial}, \eqref{eq:contra-2}, and \eqref{eq:contra-4} there exist $c_1, c_2>0$, which are independent of $T$, such that
	\[
		\norm{\Psi[f]}_{q, T}\le c_1(1+c_*^2\varepsilon_*)\|v_0\|_{\frac{N}{2}},
		\quad
		d_T(\Psi[f],\Psi[g])\le c_2 c_* \varepsilon_* d_T(f,g),
	\]
	for $f, g\in Y_T$.
	Thus, setting
	\[
		c_*
		\coloneqq
		2c_1>0,
		\quad
		\varepsilon_*
		\coloneqq
		\min
		\left\{
			\dfrac{1}{4c_1^2}, \dfrac{1}{4c_1c_2}
		\right\}>0,
	\]
	we deduce from \eqref{eq:contra-0} and \eqref{eq:contra-3} that $\Psi$ is a contraction mapping on $(Y_T, d_T)$.
	Since $\varepsilon_*>0$ and $c_*>0$ are independent of $T$, 
	there exists $v\in Y_\infty$ such that \eqref{eq:IE} holds for $t>0$.

	We next show that $v\in C([0, \infty); \mathcal{L}^{\frac{N}{2}})\cap X_{\frac{N}{2}, \infty}$ and
	\begin{equation}
	\label{eq:est_N/2}
		\norm{v}_{\frac{N}{2}, \infty}\le C\|v_0\|_{\frac{N}{2}}
	\end{equation}
	if $v\in Y_\infty$ satisfies \eqref{eq:IE} for $t>0$.
	Since $a=b=q$ and $c=N/2$ satisfy \eqref{eq:condition} and \eqref{eq:additional} in Proposition~\ref{prop:bilinear} (see Figure~\ref{fig:bilinear}-(a)), 
	we have $\Theta_2[v,v]\in X_{\frac{N}{2},\infty}$ and
	\[
		\norm{\Theta_2[v,v]}_{\frac{N}{2}, \infty}
		\le 
		C
		\norm{v}_{q,\infty}^2
		\le C\|v_0\|_{\frac{N}{2}},
		\quad
		\lim_{t\searrow0}\norm{\Theta_2[v,v]}_{\frac{N}{2}, t}
		=0.
	\]
	Since $v(t)=F(t)\ast v_0-\Theta_2[v,v](t)$, 
	we deduce from \eqref{lqlp_est;F} and Corollary~\ref{cor:semigroup}-(1), (2) 
	that $v$ belongs to $C([0, \infty); \mathcal{L}^{\frac{N}{2}})\cap X_{\frac{N}{2},\infty}$ and satisfies \eqref{eq:est_N/2}.
	
	We finally show that $v\in Y_\infty$ satisfying \eqref{eq:IE} for $t>0$ gives  $v\in X_{\infty, \infty}$ with
	\begin{equation}
	\label{eq:est_infty}
		\norm{v}_{\infty, \infty}\le C\|v_0\|_{\frac{N}{2}}.
	\end{equation}
	Let $q_1, q_2$ satisfy
	\[
		\dfrac{9}{8N}<\dfrac{1}{q_1}<\dfrac{5}{4N},
		\quad
		\dfrac{1}{2N}<\dfrac{1}{q_2}<\dfrac{3}{4N}.
	\]
	Then we observe from \eqref{eq:q} that
	the choice $a=b=q$ and $c=q_{1}$ satisfies \eqref{eq:condition} and \eqref{eq:additional} in Proposition~\ref{prop:bilinear} (see Figure~\ref{fig:bilinear}-(a)). 
	Thus, applying an argument similar to that in the previous paragraph, we obtain $v\in X_{q_1, \infty}$ 
	and $v$ satisfies
	\begin{equation}
	\label{eq:est_q}
		\norm{v}_{q_1, \infty}
		\le C\|v_0\|_{\frac{N}{2}}.
	\end{equation}
	
	By the semigroup property of $F$ we have
	\begin{equation}
	\label{eq:semi}
		v(t)
		=
		F
		\left(
			\dfrac{t}{2}
		\right)
		\ast v
		\left(
			\dfrac{t}{2}
		\right)
		-\Theta_1[v,v](t)
	\end{equation}
	for $t>0$.
	Since the choice $a=b=q_1$ and $c=q_2$ satisfies \eqref{eq:condition} in Proposition~\ref{prop:bilinear} (see Figure~\ref{fig:bilinear}-(b)), 
	we observe 
	that $\Theta_1[v,v]\in X_{q_2, \infty}$ with
	\[
			\norm{\Theta_1[v,v]}_{q_2, \infty}
			\le 
			C
			\|v_0\|_{\frac{N}{2}}.
	\]
Moreover, by \eqref{lqlp_est;F}, \eqref{eq:est_N/2}, and $v\in X_{\frac{N}{2}, \infty}$ we see that the map
	\[
		(0, \infty)
		\ni 
		t
		\mapsto 
		F
		\left(
			\dfrac{t}{2}
		\right)
		\ast v
		\left(
			\dfrac{t}{2}
		\right)
	\]
	belongs to $X_{q_2, \infty}$ and $v$ satisfies
	\[
		\Short{t}^{\frac{N}{2}(\frac{2}{N}-\frac{1}{q_2})+\frac{|\alpha|}{2}}
		\Long{t}^{\frac{N}{4}(\frac{2}{N}-\frac{1}{q_2})+\frac{|\alpha|}{4}}\left\|
			\partial^\alpha_x
			\left[
				F
				\left(
					\dfrac{t}{2}
				\right)
				\ast v
				\left(
					\dfrac{t}{2}
				\right)
			\right]
		\right\|_{q_2}
		\le
		C
		\left\|v\left(\dfrac{t}{2}\right)\right\|_{\frac{N}{2}}
		\le 
		C
		\|v_0\|_{\frac{N}{2}}
	\]
	for $t>0$ and $\alpha\in\mathbb{Z}^N_+$ with $|\alpha|\le1$. 
	Hence these estimates, together with \eqref{eq:est_q} and \eqref{eq:semi}, imply that $v\in X_{q_2, \infty}$ and 
	\begin{equation}
	\label{eq:est_q2}
		\norm{v}_{q_2, \infty}
		\le C\|v_0\|_{\frac{N}{2}}.
	\end{equation}
	Put $a=q_2$, $b=q_1$, and $c=\infty$. 
	Then these constants satisfy \eqref{eq:condition} in Proposition~\ref{prop:bilinear} (see Figure~\ref{fig:bilinear}-(c)), 
	and 
	the same method as in the proof of \eqref{eq:est_q2} yields $v\in X_{\infty, \infty}$ together with the estimate \eqref{eq:est_infty}.
	Furthermore,
	the desired estimate in $X_{r, \infty}$ for $r\in(N/2, \infty)$ follows from \eqref{eq:est_N/2}, \eqref{eq:est_infty}, 
	and the H\"{o}lder inequality. Therefore, we complete the proof of Theorem~\ref{thm:main}.
\end{proof}

%============================================================================%
\appendix
%============================================================================%

%============================================================================%
\section{Conditional uniqueness}
\label{sec:unique}
%============================================================================%
In this appendix, we prove conditional uniqueness for problem~\eqref{eq:KS}.
For simplicity, we set
\begin{alignat*}{1}
		Z_{a, T}
		\coloneqq 
		\left\{
			f 
			\in 
			C((0, T); L^a_{\rm ul}(\mathbb{R}^N))
			\, \middle| \, 
			\|
				f
			\|_{Z_{a, T}}
			\coloneqq 
			\sup_{t \in (0, T)}
			t^{\frac{N}{2}(\frac{2}{N}-\frac{1}{a})}
			\|
				f(t)
			\|_{L^a_{\rm ul}(\mathbb{R}^N)}
			<
			\infty
		\right\}
\end{alignat*}
for $a\in[N/2,\infty]$.
\begin{Proposition}
\label{prop:unique}
	Let $N\ge3$, $T>0$, 
	and 
	$
		u_0
		\in 
		\mathcal{L}^\frac{N}{2}_{\rm ul}(\mathbb{R}^N)
	$.
	Let 
	$
		u, \tilde{u}
		\in 
		C([0, T); L^\frac{N}{2}_{\rm ul}(\mathbb{R}^N))
	$ 
	be solutions to problem~\eqref{eq:KS} with 
	$
		u(0)
		=
		\tilde{u}(0)
		=
		u_0
	$.
	Moreover, 
	assume that 
	$
		u(t)
		\in
		\mathcal{L}^\frac{N}{2}_{\rm ul}(\mathbb{R}^N)
	$ 
	for 
	$
		t
		\in
		[0,T)
	$ 
	and 
	$
		u, \tilde{u} \in Z_{q, T}
	$ 
	for some 
	$
		q \in (N/2, N)
	$.
	Then 
	$
		u
		=
		\tilde{u}
	$.
\end{Proposition}
\begin{Remark}
	In view of the results in \cite{F2022, KSY2012}, 
	it seems that the additional assumption 
	$
		u, \tilde{u}
		\in 
		Z_{q, T}
	$ 
	is removable if 
	$
		N
		\ge
		4
	$.
\end{Remark}
Before proving Proposition~\ref{prop:unique}, 
we recall some properties of uniformly local Lebesgue spaces (see~\cite[Proposition~2.1, Corollary~2.2]{ARCD2004}, 
\cite[Proposition~2.2, Corollary~3.1]{MT2006}, 
and \cite[Corollary~2.3]{S2021}):
\begin{itemize}
	\item
		For 
		$
			u_0
			\in 
			L^\frac{N}{2}_{\rm ul}(\mathbb{R}^N)
		$, 
		the following conditions are equivalent:
		\begin{equation}
		\label{eq:unipro1}
			{\rm (a)}
			\ 
				u_0
				\in 
				\mathcal{L}^\frac{N}{2}_{\rm ul}(\mathbb{R}^N);\quad 
			{\rm (b)}
			\ 
			\lim_{t \searrow 0}
				\|
					G(t) \ast u_0
					-
					u_0
				\|_{L^\frac{N}{2}_{\rm ul}(\mathbb{R}^N)}
				=
				0.
		\end{equation}
	\item
		For $1 \le p \le q \le \infty$ and $\alpha \in \mathbb{Z}_+^N$ with $|\alpha|\le 1$,
		there exists 
		$
			C
			=
			C(N, p, q, \alpha)
			>
			0
		$ 
		such that
		\begin{equation}
		\label{eq:unipro2}
			\|
				\partial_x^\alpha
				G(t) \ast f
			\|_{L^q_{\rm ul}(\mathbb{R}^N)}
			\le 
			C
			t^{-\frac{N}{2}(\frac{1}{p}-\frac{1}{q})-\frac{|\alpha|}{2}}
			\|
				f
			\|_{L^{p}_{\rm ul}(\mathbb{R}^N)}
		\end{equation}
		for $t \in (0,1)$.
	\item
		Let $1<p<N$.
		Then there exists $C=C(N, p)>0$ such that
		\begin{equation}
		\label{eq:unipro3}
			\|
				\nabla K \ast f
			\|_{L^{(\frac{1}{p}-\frac{1}{N})^{-1}}_{\rm ul}(\mathbb{R}^N)}
			\le 
			C
			\|
				f
			\|_{L^{p}_{\rm ul}(\mathbb{R}^N)}.
		\end{equation}
\end{itemize}
\begin{proof}[Proof of Proposition~\ref{prop:unique}]
	Set 
	$
		w
		\coloneqq 
		u
		-
		\widetilde{u}
	$.
	Then $w$ satisfies
	\begin{equation}
	\label{eq:uni-1}
		\begin{alignedat}{1}
			w(t)
			= \, & 
			-
			\int^t_0
				\nabla \cdot 
				\left[
					G(t-s) \ast 
					(
						w(s) \nabla K \ast u(s)
					)
				\right]
			\, \mathrm{d}s\\
			& \, 
			-
			\int^t_0
				\nabla \cdot 
				\left[
					G(t-s) \ast 
					(
						\widetilde{u}(s) \nabla K \ast w(s)
					)
				\right]
			\, \mathrm{d}s
		\end{alignedat}
	\end{equation}
	for $t \in (0, T)$.
	Let $r \in (1, \infty)$ satisfy 
	$
		N/2
		<
		r
		<
		q 
		<
		N
	$.
	Since
	\[
		0
		\le 
		\dfrac{1}{r}
		\le
		\dfrac{2}{r}
		-
		\frac{1}{N}
		\le 
		1,
		\quad 
		0
		<
		\dfrac{1}{r}
		-
		\dfrac{1}{N}
		<
		\dfrac{1}{r}
		<
		1,
		\quad 
		-
		2
		+
		\frac{N}{r}
		>
		-
		1,
		\quad 
		-
		\frac{N}{2r}
		>
		-
		1,
	\]
	we observe from \eqref{eq:unipro2} and \eqref{eq:unipro3} that
	\begin{alignat*}{1}
		&
		\left\|
			\int^t_0
				\nabla \cdot 
				\left[
					G(t-s) \ast 
					(
						w(s) \nabla K \ast u(s)
					)
				\right]
			\, \mathrm{d}s
		\right\|_{L^{r}_{\rm ul}(\mathbb{R}^N)}\\
		\le \, &
		C
		\int^t_0
			(t-s)^{-\frac{1}{2}-\frac{N}{2}((\frac{2}{r}-\frac{1}{N})-\frac{1}{r})}
			\left\|
				w(s)
				[
					\nabla K \ast u(s)
				]
			\right\|_{L^{(\frac{2}{r}-\frac{1}{N})^{-1}}_{\rm ul}(\mathbb{R}^N)}
		\, \mathrm{d}s\\
		\le \, & C
		\int^t_0
			(t-s)^{-\frac{N}{2r}}
			\|
				w(s)
			\|_{L^{r}_{\rm ul}(\mathbb{R}^N)}
			\|
				\nabla K \ast u(s)
			\|_{L^{(\frac{1}{r}-\frac{1}{N})^{-1}}_{\rm ul}(\mathbb{R}^N)}
		\, \mathrm{d}s\\
		\le \, & 
		C
		\int^t_0
			(t-s)^{-\frac{N}{2r}}
			\|
				w(s)
			\|_{L^{r}_{\rm ul}(\mathbb{R}^N)}
			\|
				u(s)
			\|_{L^{r}_{\rm ul}(\mathbb{R}^N)}
		\, \mathrm{d}s\\
		\le \, & 
		C
		\|
			w
		\|_{Z_{r, t}}
		\|
			u
		\|_{Z_{r, t}}
		\int^t_0
			(t-s)^{-\frac{N}{2r}}
			s^{-2+\frac{N}{r}}
		\, \mathrm{d}s\\
		= \, &
		C
		\|
			w
		\|_{Z_{r, t}}
		\|
			u
		\|_{Z_{r, t}} 
		t^{-\frac{N}{2}(\frac{2}{N}-\frac{1}{r})}
	\end{alignat*}
	for $t \in (0, \min \{ T,1 \})$.
	By the same argument, 
	we have
	\[
		\left\|
			\int^t_0
				\nabla \cdot 
				\left[
					G(t-s) \ast 
					(
						\widetilde{u}(s) \nabla K \ast w(s)
					)
				\right]
			\, \mathrm{d}s
		\right\|_{L^r_{\rm ul}(\mathbb{R}^N)}
		\le 
		C
		\|
			w
		\|_{Z_{r, t}}
		\|
			\widetilde{u}
		\|_{Z_{r, t}} 
		t^{-\frac{N}{2}(\frac{2}{N}-\frac{1}{r})}
	\]
	for $t \in (0, \min \{ T, 1 \})$.
	Combining these estimates with \eqref{eq:uni-1}, 
	we obtain
	\begin{equation}
	\label{eq:uni-2}
		\|
			w
		\|_{Z_{r, t}}
		\le 
		C
		\left(
			\|
				u
			\|_{Z_{r, t}}
			+
			\|
				\widetilde{u}
			\|_{Z_{r, t}}
		\right)
		\|
			w
		\|_{Z_{r, t}}
	\end{equation}
	for $t \in (0, \min \{ T, 1 \})$.
	
	We next show that 
	\begin{equation}
	\label{eq:uni-3}
		\lim_{t \searrow 0}
			\|
				u
			\|_{Z_{r, t}}
		=
		\lim_{t \searrow 0}
			\|
				\widetilde{u}
			\|_{Z_{r, t}}
		=
		0.
	\end{equation}
	We only consider $u$.
	By the triangle inequality, 
	we see that
	\begin{equation}
	\label{eq:uni-4}
		\|
			u
		\|_{Z_{r, t}}
		\le 
		\|
			u
			-
			G(\cdot) \ast u_0
		\|_{Z_{r, t}}
		+
		\|
			G(\cdot) \ast u_0
		\|_{Z_{r, t}}
	\end{equation}
	for $t \in (0, T)$.
	Let $\theta \in (0,1)$ satisfy 
	$
		1/r
		=
		2
		\theta/N
		+
		(1-\theta)/q
	$.
	Then we observe from \eqref{eq:unipro2} and the H\"{o}lder inequality that
	\begin{equation}
	\label{eq:uni-5}	
	\begin{alignedat}{1}
		&
		\|
			u
			-
			G(\cdot) \ast u_0
		\|_{Z_{r, t}}\\
		\le \, &
		\|
			u
			-
			G(\cdot) \ast u_0
		\|_{Z_{\frac{N}{2}, t}}^\theta
		\|
			u
			-
			G(\cdot) \ast u_0
		\|_{Z_{q, t}}^{1-\theta}\\
		\le \, &
		C
		\left(
			\|
				u
				-
				u_0
			\|_{Z_{\frac{N}{2}, t}}
			+
			\|
				G(\cdot) \ast u_0
				-
				u_0
			\|_{Z_{\frac{N}{2}, t}}
		\right)^{\theta}
		\left(
			\|
				u
			\|_{Z_{q, t}}
			+
			\|
				u_0
			\|_{L^{\frac{N}{2}}_{\rm ul}(\mathbb{R}^N)}
		\right)^{1-\theta}
	\end{alignedat}
	\end{equation}
	for $t \in (0, T)$.
	On the other hand, 
	letting 
	$
		\{ u_{0,k} \}^\infty_{k=1}
		\subset 
		BUC(\mathbb{R}^N)
	$ 
	satisfy 
	$
		u_{0,k}
		\to 
		u_0
	$ 
	as 
	$
		k
		\to
		\infty
	$ 
	in 
	$
		L^{\frac{N}{2}}_{\rm ul}(\mathbb{R}^N)
	$, 
	we observe from \eqref{eq:unipro2} that
	\begin{alignat*}{1}
		\limsup_{t \searrow 0}
		\|
			G(\cdot)\ast u_0
		\|_{Z_{r, t}}
		\le \, & 
		\limsup_{t \searrow 0}
		\|
			G(\cdot) \ast (u_0-u_{0,k})
		\|_{Z_{r, t}}
		+
		\limsup_{t \searrow 0}
		\|
			G(\cdot) \ast u_{0,k}
		\|_{Z_{r, t}}\\
		\le \, & 
		C
		\|
			u_0
			-
			u_{0,k}
		\|_{L^{\frac{N}{2}}_{\rm ul}(\mathbb{R}^N)}
		+
		\limsup_{t \searrow 0}
		t^{\frac{N}{2}(\frac{2}{N}-\frac{1}{r})}
		\|
			u_{0,k}
		\|_{L^{r}_{\rm ul}(\mathbb{R}^N)}\\
		= \, & 
		C
		\|
			u_0
			-
			u_{0,k}
		\|_{L^{\frac{N}{2}}_{\rm ul}(\mathbb{R}^N)}
	\end{alignat*}
	for $k \in \mathbb{N}$.
	Taking $k \to \infty$, 
	we have 
	$
		\|
			G(\cdot) \ast u_0
		\|_{Z_{r, t}}
		\to 
		0
	$ 
	as 
	$t\searrow0$.
	Together with \eqref{eq:unipro1}, 
	\eqref{eq:uni-4}, 
	and \eqref{eq:uni-5}, the preceding estimate implies that \eqref{eq:uni-3} holds.
	Therefore, 
	by \eqref{eq:uni-2} and \eqref{eq:uni-3} there exists 
	$
		t_*
		\in
		(0,\min\{ T, 1 \})
	$ 
	such that 
	$
		u(t)
		=
		\widetilde{u}(t)
	$ 
	in 
	$
		L^\frac{N}{2}_{\rm ul}(\mathbb{R}^N)
	$ 
	for 
	$
		t
		\in
		[0,t_*]
	$.
	
	We finally show that 
	$
		u(t)
		=
		\widetilde{u}(t)
	$ 
	in 
	$
		L^\frac{N}{2}_{\rm ul}(\mathbb{R}^N)
	$ 
	for 
	$
		t
		\in
		[0, T)
	$.
	Let 
	\[
		T_*
		\coloneqq 
		\sup
		\left\{
			t
			\in
			(0,T)
			\, \middle| \,
			\text{
				$
					u(s)
					=
					\widetilde{u}(s)
				$ 
				in 
				$
					L^\frac{N}{2}_{\rm ul}(\mathbb{R}^N)
				$ 
				for 
				$
					s
					\in 
					[0,t)
				$
				}
		\right\}.
	\]
	By the previous argument, 
	we have $T_*>0$.
	If $T_*<T$, 
	then the continuity with respect to the time variable implies that 
	$
		u(T_*)
		=
		\widetilde{u}(T_*)
	$ 
	in 
	$
		L^\frac{N}{2}_{\rm ul}(\mathbb{R}^N)
	$.
	Since 
	$
		u(T_*) 
		\in 
		\mathcal{L}^\frac{N}{2}_{\rm ul}(\mathbb{R}^N)
	$, 
	we can apply the previous argument with $u_0$ replaced by $u(T_*)$ and hence
	\[
		u(T_*+t)
		=
		\widetilde{u}(T_*+t) 
		\quad \text{for} \quad 
		0
		<
		t
		\ll 
		1,
	\]
	which contradicts the definition of $T_*$.
	Therefore, 
	$T=T_*$, 
	which gives the desired conclusion of Proposition~\ref{prop:unique}.
\end{proof}
%============================================================================%
\section*{Conflict of Interest}
The authors declare that they have no conflict of interest.

\section*{Acknowledgements}
N. Miyake was supported in part by JSPS KAKENHI Grant Number JP24K16944.
H. Wakui was supported in part by JSPS KAKENHI Grant Number JP25K07080.
T. Yamada was supported in part by JSPS KAKENHI Grant Number JP24K06806.
%============================================================================%
\bibliography{ref}
\bibliographystyle{abbrv}
\end{document}